\documentclass[12pt]{article}
\usepackage{amsmath}
\usepackage{amssymb}
\usepackage{amsthm}
\usepackage[pdftex]{graphicx}
\usepackage{array}
\usepackage{stmaryrd}
\usepackage{bbm}
\usepackage{mathtools}
\usepackage{colonequals}
\usepackage{tikz}
\usetikzlibrary{matrix}

\usepackage{fancyhdr}
\newtheorem{thm}{Theorem}[subsection]
\newtheorem{lem}[thm]{Lemma}
\newtheorem{prop}[thm]{Proposition}
\newtheorem{conj}[thm]{Conjecture}
\newtheorem{cor}[thm]{Corollary}
\theoremstyle{definition}
\newtheorem{defn}[thm]{Definition}

\newtheorem{exmp}[thm]{Example}
\theoremstyle{remark}
\newtheorem{rem}{Remark}

\newcommand{\aff}{\mathbb{A}}

\newcommand{\cplx}{\mathbb{C}}

\newcommand{\gra}{\mathbb{G}}

\newcommand{\rats}{\mathbb{Q}}
\newcommand{\real}{\mathbb{R}}

\newcommand{\ints}{\mathbb{Z}}

\newcommand{\one}{\mathbbm{1}}

\newcommand{\al}{\alpha}

\newcommand{\del}{\delta}
\newcommand{\eps}{\epsilon}
\newcommand{\zt}{\zeta}

\renewcommand{\it}{\iota}

\newcommand{\lam}{\lambda}
\newcommand{\sg}{\sigma}

\renewcommand{\phi}{\varphi}

\newcommand{\Gam}{\Gamma}
\newcommand{\Del}{\Delta}

\newcommand{\Lam}{\Lambda}
\newcommand{\Sg}{\Sigma}
\newcommand{\Om}{\Omega}

\renewcommand{\bar}{\overline}
\renewcommand{\hom}{\text{Hom}}

\newcommand{\ten}{\otimes}
\newcommand{\dsm}{\oplus}

\newcommand{\bsl}{\backslash}

\newcommand{\w}{\wedge}
\newcommand{\inj}{\hookrightarrow}
\newcommand{\surj}{\twoheadrightarrow}

\newcommand{\isom}{\xrightarrow{\sim}}

\newcommand{\comment}[1]{}

\newcommand{\un}[1]{\underline{#1}}

\newcommand{\inn}[1]{\langle #1\rangle}
\newcommand{\ps}[1]{\llbracket #1\rrbracket}
\newcommand{\flr}[1]{\left\lfloor #1\right\rfloor}

\newcommand{\mattwo}[4]{\begin{pmatrix} #1 & #2 \\ #3 & #4\end{pmatrix}}

\DeclareFontFamily{OT1}{rsfs}{}
\DeclareFontShape{OT1}{rsfs}{n}{it}{<-> rsfs10}{}
\DeclareMathAlphabet{\mathscr}{OT1}{rsfs}{n}{it}

\newcommand{\dsc}{\mathscr{D}}

\newcommand{\hsc}{\mathscr{H}}

\newcommand{\osc}{\mathscr{O}}

\newcommand{\ssc}{\mathscr{S}}

\newcommand{\usc}{\mathscr{U}}

\newcommand{\wsc}{\mathscr{W}}

\newcommand{\zsc}{\mathscr{Z}}

\newcommand{\nfr}{\mathfrak{n}}

\newcommand{\Mfr}{\mathfrak{M}}

\newcommand{\Ufr}{\mathfrak{U}}

\newcommand{\dcal}{\mathcal{D}}

\DeclareFontFamily{U}{wncy}{}
    \DeclareFontShape{U}{wncy}{m}{n}{<->wncyr10}{}
    \DeclareSymbolFont{mcy}{U}{wncy}{m}{n}
    \DeclareMathSymbol{\Sh}{\mathord}{mcy}{"58}

\DeclareMathOperator{\cond}{cond}

\DeclareMathOperator{\diag}{diag}

\DeclareMathOperator{\id}{id}
\DeclareMathOperator{\im}{im}
\DeclareMathOperator{\ind}{Ind}
\DeclareMathOperator{\iw}{Iw}

\DeclareMathOperator{\np}{NP}

\DeclareMathOperator{\rank}{rank}

\DeclareMathOperator{\sym}{Sym}

\title{Slopes in eigenvarieties for definite unitary groups}
\author{Lynnelle Ye}

\begin{document}
\maketitle

\begin{abstract}
We generalize bounds of Liu-Wan-Xiao for slopes in eigencurves for definite unitary groups of rank $2$
to slopes in eigenvarieties for definite unitary groups of any rank. We show that for a definite unitary group of rank $n$, the Newton polygon of the characteristic power series of the $U_p$ Hecke operator has exact growth rate $x^{1+\frac2{n(n-1)}}$, times a constant proportional to the distance of the weight from the boundary of weight space. 
The proof goes through the classification of forms associated to principal series representations. We also give a consequence for the geometry of these eigenvarieties over the boundary of weight space.
\end{abstract}

\tableofcontents

\section{Introduction}

\nocite{*}

\subsection{Background and statement of main theorem}

The first ``eigenvariety'' was constructed by Coleman and Mazur in~\cite{cm98}. Now called the Coleman-Mazur eigencurve, it is a rigid analytic space parametrizing $p$-adic modular Hecke eigenforms with nonzero $U_p$-eigenvalues. Since then, further work by numerous authors has resulted in a collection of eigenvarieties for $p$-adic automorphic forms on various other groups. Particularly relevant for our purposes are the papers of Buzzard \cite{buzzard04}, \cite{buzzard07}, Chenevier \cite{chenevier04}, and Bella\"iche-Chenevier \cite{bc09}, in which eigenvarieties are eventually constructed for $p$-adic automorphic forms on definite unitary groups of all ranks. 

For simplicity of notation in this introduction, let $p$ be an odd prime. 
We write $v$ for the $p$-adic valuation and $|\cdot|$ for the $p$-adic norm, normalized so that $v(p)=1$ and $|p|=p^{-1}$. A \emph{weight} of a $p$-adic modular form is a continuous character of $\ints_p^\times$, and the \emph{weight space} is the rigid analytic space $\wsc$ parametrizing such characters. The $T$-coordinate of a point $w\in\wsc$ is the value $T(w)=w(\exp(p))-1$; the space $\wsc$ turns out to be a disjoint union of $p-1$ open unit discs with parameter $T$. For $r\in(0,1)$, we write $\wsc_{>r}$ for the rigid analytic subset of $\wsc$ where $|T|>r$.

We fix a tame level and let $\zsc$ be the corresponding eigencurve. We let $w:\zsc\to\wsc$ be the map taking an eigenform to its weight, $a_p:\zsc\to\gra_m$ be the map taking an eigenform to its $U_p$-eigenvalue, and $\zsc_{>r}$ be the preimage of $\wsc_{>r}$ in $\zsc$. The following conjecture, sometimes called the ``halo conjecture'', describes the geometry of the part of the eigencurve lying over the ``boundary'' of weight space (i.e. $\zsc_{>r}$ for $r$ sufficiently close to $1$).

\begin{conj}[Coleman-Mazur-Buzzard-Kilford, as stated by Liu-Wan-Xiao \cite{lwx17}]
\label{cmbk}
When $r\in(0,1)$ is sufficiently close to $1^-$, the following statements hold.
\begin{enumerate}
\item The space $\zsc_{>r}$ is a disjoint union of (countably infinitely many) connected components $Z_1,Z_2,\dotsc,$ such that the weight map $w:Z_n\to\wsc_{>r}$ is finite and flat for each $n$.

\item\label{tend0} There exist nonnegative rational numbers $\al_1,\al_2,\dotsc\in\rats$ in non-decreasing order and tending to infinity such that, for each $n$ and each point $z\in Z_n$, we have
\[
|a_p(z)|=|T(w(z))|^{\al_n}.
\]
\end{enumerate}
\end{conj}

Note that Part~\ref{tend0} of Conjecture~\ref{cmbk} implies that as one approaches the boundary, the slope $v(a_p(z))$ approaches $0$ in proportion to $v(T(w(z)))$.

Liu, Wan, and Xiao (\cite{lwx17}) proved the equivalent version of this conjecture for automorphic forms on definite quaternion algebras over $\rats$. The key step in their work is to obtain strong upper and lower bounds on the Newton polygon of the characteristic power series of the $U_p$-operator. For consistency with our discussion, we will describe these bounds in the context of rank-$2$ definite unitary groups over $\rats$, for which the analysis is exactly the same. 

Let $G$ be an algebraic group over $\rats$ such that $G(\real)\cong U_n(\real)$ and $G(\rats_p)\cong GL_n(\rats_p)$, and $\usc\subset G(\aff_f)$ a compact open subgroup satisfying minor technical conditions. The corresponding eigenvariety $\zsc$ is now a rigid analytic space of dimension $n-1$ lying over the weight space $\wsc$ parametrizing continuous characters of $(\ints_p^\times)^{n-1}$. This $\wsc$ is a disjoint union of $(p-1)^{n-1}$ open unit polydiscs of dimension $n-1$ with parameters $T_1,\dotsc,T_{n-1}$. Let $\ssc_w(G,\usc)$ be the space of $p$-adic automorphic forms on $G$ of weight $w$ and level $\usc$. Then Liu-Wan-Xiao show that when $n=2$, the Newton polygon of $\det(I-XU_p|\ssc_w(G,\usc))$ is shaped approximately like the curve $y=Av(T_1(w))x^2$, where $A$ is a constant depending only on $G$, $\usc$, and $p$.

In this paper, we generalize this bound to definite unitary groups of all ranks by showing that for arbitrary $n$, the Newton polygon of $\det(I-XU_p|\ssc_w(G,\usc))$ is shaped approximately like $y=Av(T_i(w))x^{1+\frac2{n(n-1)}}$, assuming that the $v(T_i(w))$s are not extremely different in size. A more precise statement follows.

\begin{thm}
\label{mybounds}
\begin{enumerate}
\item \label{mylbd} There are constants $A_1,C>0$ (depending only on $G$, $\usc$, and $p$) such that for all $w$ such that each $|T_i(w)|>\frac1p$, the Newton polygon of the power series $\det(I-XU_p|\ssc_w(G,\usc))$ lies above the curve 
\[
y=\left(A_1x^{1+\frac2{n(n-1)}}-C\right)\min_iv(T_i(w)).
\]

\item \label{myubd} Suppose that $w(a_1,\dotsc,a_{n-1})=\prod_i a_i^{t_i}\chi_i(a_i)$, where $(t_1,\dotsc,t_{n-1})\in(\ints_{\ge0})^{n-1}$ with $t_1\ge\dotsb\ge t_{n-1}$, and each $\chi_i$ is a finite character of conductor $c_i$ such that $\cond(\chi_i\chi_j^{-1})=\max(\cond(\chi_i),\cond(\chi_j))$ for all $i\neq j$. 
Let $\chi_{(1)},\dotsc,\chi_{(n-1)}$ be the characters $\chi_1,\dotsc,\chi_{n-1}$ reordered so that $\cond(\chi_{(1)})\le\cond(\chi_{(2)})\le\dotsb\le\cond(\chi_{(n-1)})$, let $c_{(i)}=\cond(\chi_{(i)})$, and let $T_{(i)}=T(\chi_{(i)})$.

Then there is a constant $h$ (depending only on $G$, $\usc$, and $p$), a polynomial $d_{t_1,\dotsc,t_{n-1}}$ of total degree $\frac{n(n-1)}2$ in the $t_i$s, and a linear function $l(t_1,\dotsc,t_{n-1})$ such that the Newton polygon of $\det(I-XU_p|\ssc_w(G,\usc))$ contains at least 
\[
hp^{c_{(1)}+2c_{(2)}+\dotsb+(n-1)c_{(n-1)}-\frac{n(n-1)}2}d_{t_1,\dotsc,t_{n-1}}
\] 
segments of slope at most $l(t_1,\dotsc,t_{n-1})$, hence passes below the point
\[
\left(hp^{c_{(1)}+2c_{(2)}+\dotsb+(n-1)c_{(n-1)}-\frac{n(n-1)}2}d_t,hp^{c_{(1)}+2c_{(2)}+\dotsb+(n-1)c_{(n-1)}-\frac{n(n-1)}2}d_tl(t)\right).
\]
If $t_i-t_{i+1}\ge\eps(t_j-t_{j+1})$ for all $i\neq j$, this point can also be written as
\[
\left(x,A_2\left(v(T_{(1)})^{\frac{2}{n(n-1)}}v(T_{(2)})^{\frac{2\cdot 2}{n(n-1)}}\dotsb v(T_{(n-1)})^{\frac{2\cdot(n-1)}{n(n-1)}}\right)x^{1+\frac2{n(n-1)}}\right)
\]
for $x=hp^{c_{(1)}+2c_{(2)}+\dotsb+(n-1)c_{(n-1)}-\frac{n(n-1)}2}d_t$ and a constant $A_2$ (depending additionally on $\eps$). Note that in particular,
\[
v(T_{(1)})^{\frac{2}{n(n-1)}}v(T_{(2)})^{\frac{2\cdot 2}{n(n-1)}}\dotsb v(T_{(n-1)})^{\frac{2\cdot(n-1)}{n(n-1)}}\le\max_i v(T_i).
\]
\end{enumerate}
\end{thm}

\begin{rem}
It will be evident that the condition that $\cond(\chi_i\chi_j^{-1})=\max(\cond(\chi_i),\cond(\chi_j))$ for all $i\neq j$ is not required for the proof to go through; it is just there to allow us to state the best and cleanest bound.
\end{rem}

We also leverage Theorem~\ref{mybounds} to prove two statements that may be more geometrically satisfying. First, we prove the following alternative version of the upper bound which provides infinitely many upper bound points on the same Newton polygon.

\begin{thm}
\label{unified-ubd}
Suppose that $w(a_1,\dotsc,a_{n-1})=\prod_i a_i^{t_i}\chi_i(a_i)$, where $(t_1,\dotsc,t_{n-1})\in(\ints_{\ge0})^{n-1}$ with $t_1\ge\dotsb\ge t_{n-1}$, and each $\chi_i$ is a finite character of conductor $c_i$ such that $\cond(\chi_i\chi_j^{-1})=\max(\cond(\chi_i),\cond(\chi_j))$ for all $i\neq j$. 
Then there is a constant $A_2$ such that for every radius $r>0$, there is a weight $s$ such that $|T_i(w)-T_i(s)|<r$ for all $i$, $|T_i(\chi t)|=|T_i(s)|$ for all $i$, and the Newton polygon of $\det(I-XU_p|\ssc_w(G,\usc))$ lies below an infinite sequence of points (with $x$-coordinates going to $\infty$) lying on the curve parametrized by
\[
\left(x, A_2x^{1+\frac1{\binom n2}}\cdot\left(v(T_{(1)}(s))^{\frac{2}{n(n-1)}}v(T_{(2)}(s))^{\frac{2\cdot 2}{n(n-1)}}\dotsb v(T_{(n-1)}(s))^{\frac{2\cdot(n-1)}{n(n-1)}}\right)\right).
\]
\end{thm}

We then use the lower bound of Theorem~\ref{mybounds} to prove the following (vaguely stated) decomposition result for the boundary of the eigenvariety.

\begin{thm}
\label{disconnect-intro}
Let $a_p:\zsc\to\gra_m$ be the map taking a point of $\zsc$ to its $U_p$-eigenvalue. For $\al\in\real_{\ge0}$, let $\zsc(\al)$ be the subset of points $z\in\zsc$ such that $v(a_p(z))=\al v(T_i(z))$. Then over certain open subsets of the weight polydisc boundary where $v(T_i)$ is much smaller than all the other $v(T_j)$s, $\zsc(\al)$ is disconnected from its complement in $\zsc$.
\end{thm}

(See Theorem~\ref{disconnect} for the precise statement.) Theorems~\ref{mybounds}, \ref{unified-ubd}, and \ref{disconnect-intro} can all be seen as weak generalizations of Conjecture~\ref{cmbk} to definite unitary groups of arbitrary rank. 

Historically, Conjecture~\ref{cmbk} arose from a question of Coleman and Mazur \cite{cm98} and was suggested by a computation of Buzzard and Kilford \cite{bk05} for $p=2$ and tame level $1$. Further explicit computations for small primes were later done by Roe~\cite{roe14}, Kilford~\cite{kilford08}, and Kilford and McMurdy~\cite{km12}. The conjecture is given above in the form stated by Liu, Wan, and Xiao \cite{lwx17}, whose proof for definite quaternion algebras builds on the work of Wan, Xiao, and Zhang \cite{wxz17}. Statements of this nature can have far-reaching consequences for the arithmetic of modular forms---see for example \cite{jn19} or \cite{nt19}. 

As far as we know, there is little prior work on the shape of the Newton polygon of $\det(I-XU_p|\ssc_w(G,\usc))$ for any $G$ of rank greater than $2$. The only prior result for general rank we have been able to find in the literature is Chenevier's lower bound in~\cite{chenevier04} of the form $y=Ax^{1+\frac1{2^n-n-1}}$, which 
applies only to the center of weight space. There is also a similar lower bound for Hilbert modular forms over real quadratic fields by Birkbeck \cite{birkbeck19}. As late as 2018, Andreatta, Iovita, and Pilloni wrote~\cite{aip15} that there were not even any conjectures about the actual shape of the Newton polygon for higher-dimensional eigenvarieties in the literature.


\subsection{Proof outline}

The proof of Part~\ref{mylbd} of Theorem~\ref{mybounds} is an application of the method of Johansson-Newton \cite{jn16}. They construct families of automorphic forms extending \emph{over} the boundary of weight space, to points in what can be viewed as an adic compactification of weight space, and show that the eigenvariety also extends to those points. (Also see Gulotta \cite{gulotta19} for an analogous construction extending equidimensional eigenvarieties.) Consequently, we can compute the matrix coefficients of $U_p$ in an explicit basis for the space of forms over the ``boundary weights'' given by monomials in the matrix coefficients of the dimension-$\frac{n(n-1)}2$ maximal lower unipotent subgroup of $GL_n(\rats_p)$. Explicit bounds on those matrix coefficients arise directly from the proof of complete continuity of $U_p$.

The proof of Part~\ref{myubd} of Theorem~\ref{mybounds} requires a detailed analysis of $p$-adic automorphic representations which may be of independent interest. As in the proof of Proposition 3.20 of \cite{lwx17}, 
we would like to carry out the following steps:

\begin{enumerate}
\item \label{subspace} Construct a subspace $S$ of $\ssc_w(G,\usc)$ of dimension $hp^{c_{(1)}+2c_{(2)}+\dotsb+(n-1)c_{(n-1)}-\frac{n(n-1)}2}d_t$ which can be thought of as the space of ``classical forms of weight $w$ and minimal level''.

\item \label{upinjective} Prove that $U_p$ is injective on $S$, so that all eigenforms in $S$ have finite slope.

\item \label{finiteissmall} Prove that finite-slope eigenforms in $S$ have slope bounded above by $l(t)$.
\end{enumerate}

In general, it is a fact that a classical form $f$ is finite-slope if and only if the local component $\pi_{f,p}$ at $p$ of its associated automorphic representation is a principal series representation of $GL_n(\rats_p)$. For $n=2$ as in Liu-Wan-Xiao, the upper bound then follows from the fact that $\pi_{f,p}$ is a principal series if the level of $f$ equals the conductor of its central character, which can be checked (as in Loeffler-Weinstein \cite{lw12}) by comparing its level to that of the new vectors in each of the three possible Bernstein-Zelevinsky classes of representations (principal series, special, and supercuspidal), those new vectors having been written down by Casselman \cite{casselman73}. The dimension of the space of such $f$ is easy to count. 

To detect when $\pi_{f,p}$ is a principal series for all $n$ without brute-forcing through Bernstein-Zelevinsky classes, we use Roche's analysis of principal series types (\cite{roche98}). (As a historical note, much of the analysis we rely on was already done for $GL_n$ by Howe and Moy---see e.g. \cite{hm90}. For more information on types in general, see Fintzen \cite{fintzen18}.) For a smooth character $\chi$ of $T(\rats_p)$, Roche gives a subgroup $J\subset GL_n(\ints_p)$ and an extension of $\chi|_{T(\ints_p)}$ to $J$ such that an irreducible representation $\pi$ of $GL_n(\rats_p)$ is a principal series associated to an unramified twist of $\chi$ if and only if $\pi$ contains a vector on which $J$ acts by $\chi$ (which we will call a $(J,\chi)$-vector). 

To proceed, we construct a space of forms $S$ so that for any eigenform $f\in S$, $\pi_{f,p}$ admits a nontrivial map from $\ind_J^{\iw_p}\chi$, where $\iw_p$ is the subgroup of $GL_n(\ints_p)$ of matrices that are upper triangular mod $p$, hence contains a $(J,\chi)$-vector and is a principal series. The $S$ we construct is not a priori a subspace of $\ssc_w(G,\usc)$, but we can show that it embeds into $\ssc_w(G,\usc)$ using Emerton's locally analytic Jacquet functor. 
The dimension of $S$ is proportional to the product of $\dim\ind_J^{\iw_p}\chi$, which is a function of the valuations $v(T_i)$, and the dimension of the algebraic representation of $GL_n$ of highest weight corresponding to the algebraic part of $(T_1,\dotsc,T_{n-1})$, which is a polynomial of total degree $\frac{n(n-1)}2$ in the weight parameters by a combinatorial calculation. This covers Steps \ref{subspace} and \ref{upinjective}. Then we do Step \ref{finiteissmall} by constructing companion forms $f^w$ of $f$ for each $w\in S_n$ such that the slopes of all the companion forms sum to $l(t)$.

We can also use $S$ to help understand the standard classical subspaces of $\ssc_w(G,\usc)$. Specifically, when the pair $(J,\chi)$ arising from the weight $w$ satisfies the additional technical condition that $\cond(\chi_i)<2\cond(\chi_j)$ for all $i\neq j$ with $i,j\neq n$, we show by a Mackey theory calculation that $\ind_J^{\iw_p}\chi$ is an irreducible representation of $\iw_p$. In this case, $S$ can be embedded into a classical subspace of $\ssc_w(G,\usc)$, and we can show the following.

\begin{thm}
\label{classifyps-vague}
$S$ is \emph{precisely} the space of finite-slope classical forms of weight $w$.
\end{thm}

To do this, we slightly refine the setup of the Bella\"iche-Chenevier construction of the eigenvariety in order to precisely define the sense in which $S$ is ``minimal level''. Consequently, our upper bound is the best possible with existing methods, except possibly for the size of $l(t)$. 

\section{Organization}

In Section~\ref{construction}, we describe the construction of the eigenvarieties we are interested in, primarily following Chenevier~(\cite{chenevier04}) and Bella\"iche-Chenevier~(\cite{bc09}), adding some extra details in places of particular importance to us. For example, we give a slightly more general definition of local analyticity of $p$-adic automorphic forms which allows different radii of analyticity for different coordinates and prove that it works, which aids in proving Theorem~\ref{classifyps-vague}.

In Section~\ref{localgweights}, we analyze the subspaces of classical automorphic forms of locally algebraic weights and the automorphic representations they generate, thus carrying out Steps \ref{subspace} and \ref{upinjective} above, and proving Theorem~\ref{classifyps}, a precise version of Theorem~\ref{classifyps-vague}.

In Section~\ref{boundproofs}, we carry out Johansson-Newton's method and Step \ref{finiteissmall} to prove Theorem~\ref{mybounds} and Theorem~\ref{unified-ubd}.

Finally, in Section~\ref{geometry}, we state and prove a precise version of Theorem~\ref{disconnect-intro} and discuss other geometric consequences of Theorem~\ref{mybounds}. Unlike in Liu-Wan-Xiao's setting, for higher-dimensional eigenvarieties, the lower and upper bounds do not match at any point on the Newton polygon, and we cannot expect them to, because there exist (probably) non-classical forms of slopes smaller than some classical forms. As a result, we cannot prove the equivalent of Conjecture~\ref{cmbk} for these higher-dimensional eigenvarieties. However, we can prove that certain boundary sections of the eigenvariety decompose into many disconnected components (with the caveat that we cannot verify that those sections are nonempty, although in fact we expect them to be everything).

\subsection*{Acknowledgments}

First and foremost, I would like to thank Mark Kisin for suggesting this problem and providing years of advice about it and math in general. I am also grateful to Tasho Kaletha for pointing me to Roche's work on principal series types, Ju-Lee Kim for clearing up my confusion about new vectors, Jessica Fintzen for improving Proposition~\ref{manyps} below (among other useful suggestions), and Yifeng Liu for helpful comments on a draft. I would also like to thank Jo\"el Bella\"iche, John Bergdall, Christopher Birkbeck, Erick Knight, Koji Shimizu, David Yang, Zijian Yao, and Yihang Zhu for explaining various facts about eigenvarieties and representation theory to me, and George Boxer, Frank Calegari, Matt Emerton, Benedict Gross, Michael Harris, Guy Henniart, Maxim Jeffs, David Loeffler, Barry Mazur, Naomi Sweeting, Richard Taylor, Pei-Yu Tsai, Eric Urban, Jared Weinstein, Yujie Xu, and Rong Zhou for various other helpful conversations.

This work was partially done during the support of the National Defense Science and Engineering Graduate Fellowship, and writing was completed during the support of the National Science Foundation Mathematical Sciences Postdoctoral Research Fellowship. 

\section{Bella\"iche-Chenevier eigenvarieties for definite unitary groups}
\label{construction}

Let $p$ be a prime. Let $q=4$ if $p=2$ and $q=p$ otherwise. In this section, we go through the construction of eigenvarieties for definite unitary groups in the language of Chenevier and Bella\"iche-Chenevier. In Section~\ref{padicautforms}, we define the groups and the spaces of $p$-adic automorphic forms we are interested in, notably including the spaces of classical forms whose interpolation was the original motivation for this construction. In Section~\ref{weightspace}, we describe the properties of the space of $p$-adic weights. In Section~\ref{plucker}, we define certain coordinates on spaces of functions on the Iwahori subgroup $\iw_p$, using a convenient hybrid of the language of Chenevier and Bella\"iche-Chenevier. In Section~\ref{subgpnotation}, we introduce systematic notation for certain subgroups of $\iw_p$. In Section~\ref{sheafpadicautforms}, we define the space of families of $p$-adic automorphic forms over weight space, along with subspaces of locally analytic families. In Section~\ref{upa}, we define $U_p$-operators and work through their various important properties in great detail. In Section~\ref{spectralvariety}, we define the desired eigenvarieties.

\subsection{$p$-adic automorphic forms}
\label{padicautforms}

Let $E$ be an imaginary quadratic field over $\rats$, and $D$ a central simple $E$-algebra of rank $n^2$ which has an involution $x\mapsto x^*$ extending the nontrivial automorphism $\sg$ of $E$ over $\rats$ (for example, $D$ could be $GL_n(E)$, in which case $x^*$ would be $\sg(x)^T$). Let $G/\rats$ be the group whose $R$-points, for a $\rats$-algebra $R$, are
\[
G(R)=\{x\in D\ten_\rats R\mid xx^*=1\}.
\]
$G(\rats_p)$ is isomorphic to $GL_n(\rats_p)$ if $p$ is split in $E$ (since then $E_p\cong\rats_p\dsm\rats_p$ with $\sg$ switching factors) and to $U_n(\rats_p)$ if $p$ is inert in $E$; we will assume that $p$ is split in $E$ and $G(\rats_p)\cong GL_n(\rats_p)$. Also, $G(\real)\cong U_{s,t}(\real)$ for $(s,t)$ the signature of $Q(x)=xx^*$; we will assume that $Q(x)$ has signature $(n,0)$ or $(0,n)$, so that $G(\real)$ is compact.

As usual, we write $B$ and $\bar{B}$ for the upper and lower triangular Borel subgroups of $GL_n$ respectively, $T$ for the diagonal torus, and $N$ and $\bar{N}$ for the upper and lower unipotent subgroups of $GL_n$ respectively.

Write $\aff=\aff_\rats$, $\aff_f$ for the finite adeles of $\aff$, and $\aff_f^p$ for the finite adeles trivial at $p$. Let $\usc$ be a compact open subgroup of $G(\aff_f)$ of the form $\usc_p\times \usc^p$, where $\usc_p$ is a compact open subgroup of $G(\rats_p)$ (called the wild level structure) and $\usc^p$ a compact open subgroup of $G(\aff_f^p)$ (called the tame level structure). We can now define $V$-valued automorphic forms for any $\usc_p$-module $V$.

\begin{defn}
If $V$ is a $k[\usc_p]$-module for any field $k$, write $V(G,\usc)$ for the $k$-vector space of maps
\[
f:G(\rats)\bsl G(\aff_f)\to V
\]
such that $f(xu)=u_p^{-1}f(x)$ for all $x\in G(\rats)\bsl G(\aff_f)$ and $u\in\usc$. Equivalently,
\[
V(G,\usc)=(\hom_{set}(G(\rats)\bsl G(\aff_f),k)\ten_k V)^{\usc}
\]
where the action of $\usc$ on $\hom_{set}(G(\rats)\bsl G(\aff_f), k)$ is right translation and the action on $V$ is through $\usc_p$. For any submonoid $\usc'\supseteq\usc$ of $G(\aff_f)$ which has an action on $V$ that is trivial for $\usc$, $V(G,\usc)$ is a $\usc'$-module with action $(uf)(x)=u_pf(xu)$.
\end{defn}

We will frequently express examples using the following notation: if $B\subseteq H$ are groups, $R$ is a ring, and $s:B\to R$ is a character, let
\[
\ind_B^{H}s=\{f:H\to R\mid f(hb)=s(b)f(h)\text{ for all }h\in H,b\in B\},
\]
and if $P$ is a property of some functions $f\in\ind_B^Hs$ which is invariant under left translation by $H$, let
\[
\ind_B^{H,P}s=\{f\in\ind_B^Hs\mid f\text{ has property }P\}.
\]
Then $\ind_B^{H,P}s$ is an $R$-module with a (left) action of $H$ given by $(hf)(x)=f(h^{-1}x)$ for all $h,x\in H$. Note that our left/right conventions for induction are unusual for convenience.

For example, if $k$ is a field, $t=(t_1,\dotsc,t_n)\in\ints^n$, and we write $\diag(d_1,\dotsc,d_n)$ for the diagonal matrix with entries $d_1,\dotsc,d_n$ along the diagonal, we can interpret $t$ as the character of the diagonal torus $T(k)$ of $GL_n(k)$ taking $\diag(d_1,\dotsc,d_n)$ to $\prod_{i=1}^n d_i^{t_i}$, and thus as the character of the upper triangular Borel $B(k)$ 
obtained by reducing to $T(k)$ and applying $t$. In the event that $t_1\ge\dotsb\ge t_n$, the $k$-vector space 
\[
\ind_{B(k)}^{GL_n(k),alg}t,
\]
where $alg$ stands for \emph{algebraic} (i.e. $f:GL_n(k)\to k$ comes from an element of $k[GL_n]$), is the irreducible algebraic representation of $GL_n$ over $k$ of highest weight $t$ (See Section 12.1.3 of \cite{gw09} and Proposition 2.2.1 of \cite{chenevier04}). We call this representation $S_t(k)$. Then $S_t(k)(G,\usc)$ is the space of classical $p$-adic automorphic forms on $G$ of weight $t$ and level $\usc$ with coefficients in $k$. 

One way to picture $V(G,\usc)$ is as follows. By the generalized finiteness of class groups (see Theorem 5.1 of \cite{borel63}), the set $G(\rats)\bsl G(\aff_f)/\usc$ is finite. Fix double coset representatives $x_1,\dotsc,x_h\in G(\aff_f)$. Then we have an isomorphism
\begin{align*}
V(G,\usc) & \isom  \bigoplus_{i=1}^h V^{x_i^{-1}G(\rats)x_i\cap\usc} \\
f & \mapsto  (f(x_1),\dotsc,f(x_h)).
\end{align*}

Because $G(\real)\cong U_n(\real)$ is compact, $G(\rats)$ is discrete in $G(\aff_f)$ (see e.g. Proposition 1.4 of~\cite{gross99} or Proposition 3.1.2 of~\cite{loeffler10}). 
Since in addition $\usc$ is compact, the group $x_i^{-1}G(\rats)x_i\cap\usc$ is always finite, and it is trivial if $\usc^p$ is sufficiently small. (For example, by Proposition 4.1.1 of~\cite{chenevier04}, there is an integer $e_n$ depending only on $n$ such that $x_i^{-1}G(\rats)x_i\cap\usc$ is guaranteed to be trivial if the image of $\usc^p$ in $G(\rats_l)$ is contained in $\Gam(l)=\{g\in GL_n(\ints_p)\mid g\equiv1\pmod{l}\}$ for some prime $l\nmid e_n$.) It is this fact that makes the construction of the eigenvariety for $G$ so sleek. 

When convenient, we will assume that $\usc^p$ is sufficiently small (sometimes called ``neat'' in the literature) and thus $V(G,\usc)\cong V^h$. As in Remark 2.14 of \cite{lwx17}, this does not affect our results, because the eigenvariety for any $\usc^p$ is a union of connected components of the eigenvariety for a sufficiently small subgroup of $\usc^p$. 

\subsection{Weight space}
\label{weightspace}

A \emph{weight} is a continuous character of $T(\ints_p)\cong(\ints_p^\times)^n$. Such a weight can be viewed as a character of $B(\ints_p)$ by reduction to $T(\ints_p)$. (In the introduction, we defined a weight instead to be a character of $(\ints_p^\times)^{n-1}$, that is, a character of $T(\ints_p)$ that is trivial on the last $\ints_p^\times$-factor. We will go back to restricting possible weights to the subset that is trivial on the last $\ints_p^\times$-factor whenever it is convenient, because any character of $T(\ints_p)$ can be twisted by a central character to one in this restricted subset, and central characters do not change spaces of automorphic forms in an interesting way.) 

The \emph{weight space} $\wsc^n$ is the rigid analytic space over $\rats_p$ such that for any affinoid $\rats_p$-algebra $A$, $\wsc^n(A)$ is the set of continuous characters $(\ints_p^\times)^n\to A^\times$. Let $\Del^n=((\ints/q\ints)^\times)^n$. We have
\[
(\ints_p^\times)^n\cong \Del^n\times(1+q\ints_p)^n
\]
so an $A$-point of $\wsc^n$ is determined by a character of $\Del^n$ and a character of $(1+q\ints_p)^n$. Furthermore, a character $s$ of $(1+q\ints_p)^n$ is determined by the values $T_i(s)=s(1,\dotsc,1,\exp(q),1,\dotsc,1)-1$ (where the $i$th entry is $\exp(q)$ and all the others are $1$), since $\exp(q)$ topologically generates $1+q\ints_p$. By Lemma 1 of~\cite{buzzard04}, the coordinates $(T_1,\dotsc,T_n)\in A^n$ come from an $A$-point of $\wsc^n$ precisely when they are topologically nilpotent. Thus $\wsc^n$ can be pictured as a finite disjoint union of $\phi(q)$ open unit polydiscs with coordinates $(T_1,\dotsc,T_n)$, one for each tame character of $\Del^n$. 

We use 
\[
[\cdot]:(\ints_p^\times)^n\to\ints_p\ps{(\ints_p^\times)^n}
\]
to denote the universal character of $(\ints_p^\times)^n$ and $\Lam^n$ to denote the Iwasawa algebra 
\[
\Lam^n=\ints_p\ps{(\ints_p^\times)^n}\cong\ints_p[\Del^n]\ten_{\ints_p}\ints_p\ps{(1+q\ints_p)^n}\cong\ints_p[\Del^n]\ten_{\ints_p}\ints_p\ps{T_1,\dotsc,T_n}
\]
where $T_i=[(1,\dotsc,1,\exp(q),1,\dotsc,1)]-1$ with the $\exp(q)$ in the $i$th position; then continuous homomorphisms $\chi:\Lam^n\to A$ are in bijection with $A$-points of $\wsc^n$ via $\chi\mapsto\chi\circ[\cdot]$. 

\begin{exmp}[dominant algebraic weights]
If $t_1\ge\dotsb\ge t_n$ are integers, the algebraic character $(d_1,\dotsc,d_n)\mapsto\prod_{i=1}^n d_i^{t_i}$ is a $\rats_p$-point of $\wsc^n$ with $T$-coordinates
\[
(\exp(t_1q)-1,\dotsc,\exp(t_nq)-1)
\]
such that the valuation of $T_i=\exp(t_iq)-1$ is
\[
v\left(\left(1+t_iq+\frac{(t_iq)^2}{2!}+\dotsb\right)-1\right)=v(t_iq).
\]
We remark that the weight polydisc in which this character appears is determined by $(t_1,\dotsc,t_n)\pmod{\phi(q)}$. 
\end{exmp}

If $\chi:\ints_p\to\cplx_p^\times$ is a finite-order character, we will borrow the following slightly nonstandard definition of the conductor $\cond(\chi)$ of $\chi$ from Section 3 of~\cite{roche98}: it is the least \emph{positive} integer $n$ such that $1+p^n\ints_p\subset\ker(\chi)$. Thus the conductor of the trivial character is $1$ but the conductor of any other character is the same as with the usual definition.

\begin{exmp}[locally algebraic weights]
\label{localgtcoords}
If $t_1\ge\dotsb\ge t_n$ are integers and $\chi_1,\dotsc,\chi_n$ are finite-order characters $\ints_p^\times\to\cplx_p^\times$, the ``locally algebraic'' character 
\[
(d_1,\dotsc,d_n)\mapsto\prod_{i=1}^n\chi_i(d_i)d_i^{t_i}
\] 
is a $\cplx_p$-point of $\wsc^n$. If $\chi_i$ is nontrivial with conductor $c$, we have
\[
v(T_i)=v(\chi_i(\exp(q))\exp(t_iq)-1)=\begin{cases}
v(t_iq) & \text{ if }p>2 \text{ and }c_i=1 \\
\frac{q}{p^{c_i-1}(p-1)} & \text{ if } p>2 \text{ and } c_i\ge2 \\
v(t_iq) & \text{ if }p=2 \text{ and }c_i=3 \\
\frac{q}{p^{c_i-1}(p-1)}=\frac1{2^{c_i-3}} & \text{ if }p=2 \text{ and }c_i\ge4.
\end{cases}
\]
For completeness, we quickly prove the second case above; the others are similar. The value $\chi_i(\exp(q))=\chi_i(\exp(p))$ is a primitive $p^{c_i}$th root of unity, say $\zt_{p^{c_i}}$. Let
\[
f(X)=\frac{X^{p^{c_i}}-1}{X^{p^{c_i-1}}-1}=\prod_{a\in(\ints/p^{c_i}\ints)^\times}(X-\zt_{p^{c_i}}^a)=X^{p^{c_i}-p^{c_i-1}}+X^{p^{c_i}-2p^{c_i-1}}+\dotsb+X^{p^{c_i-1}}+1.
\]
Then $f(1)=p=\prod_{a\in(\ints/p^{c_i}\ints)^\times}(1-\zt_{p^{c_i}}^a)$. Each term in the product has the same valuation, since they are Galois conjugate, and there are $p^{c_i-1}(p-1)$ such terms. So $v(\chi_i(\exp(q))-1)=\frac1{p^{c_i-2}(p-1)}$. The factor of $\exp(t_iq)$ has no effect since it is $1\pmod{p}$.
\end{exmp}

In general, if $A$ is a Banach $\rats_p$-algebra, we say that a character $s:\ints_p^\times\to A^\times$ is $c$-locally analytic if its restriction to $1+p^c\ints_p$ is given by a convergent power series with coefficients in $A$. Every continuous character $s$ is $c$-locally analytic for some $c$: let $T=s(\exp(q))-1$ and choose $c$ such that $|T^{q^{-1}p^c}|<q^{-1}$. Then we have
\[
s(z)=s\left(\exp(q)^{\frac1{q}\log z}\right)=[(1+T)^{q^{-1}p^c}]^{\frac1{p^c}\log z}=[1+((1+T)^{q^{-1}p^c}-1)]^{\frac1{p^c}\log z}
\]
if this converges. But by our choice of $c$, we have $|(1+T)^{q^{-1}p^c}-1|<q^{-1}$, and if $z\in(1+p^c\ints_p)$ then $\left|\frac1{p^c}\log z\right|\le1$. By Lemma 3.6.1 of~\cite{chenevier04}, this expression is a convergent power series in $z$. 


Naturally, if $s:(\ints_p^\times)^n\to A^\times$ is a character, we say that it is $(c_1,\dotsc,c_n)$-locally analytic if it is $c_i$-locally analytic in the $i$th factor.

If $W$ is any open affinoid subset of $\wsc$, we use 
\[
[\cdot]_W:(\ints_p^\times)^n\to\osc(W)^\times
\]
to denote the universal character of $(\ints_p^\times)^n$ with coefficients in $\osc(W)$. Note that $[\cdot]_W$ is $(c_1,\dotsc,c_n)$-locally analytic with $c_i$ depending on $\max_{s\in W(\cplx_p)}|T_i(s)|$. 

\subsection{Coordinates on spaces of functions on $\iw_p$}
\label{plucker} 

If $A$ is an affinoid $\rats_p$-algebra and $s:(\ints_p^\times)^n\to A^\times$ is a weight, we can view any function $f\in\ind_{B(\ints_p)}^{\iw_p}s$ as a function on $\ints_p^{n(n-1)/2}$ by restricting $f$ to the lower unipotent subgroup $\bar{N}$ and applying the map
\begin{align*}
\ints_p^{n(n-1)/2} &\to \bar{N} \\
\un{z}=(z_{ij}) &\mapsto \bar{N}(\un{z})=
\begin{pmatrix}
1 & 0 & 0 & \dotsb & 0 \\
pz_{21} & 1 & 0 & \dotsb & 0 \\
pz_{31} & pz_{32} & 1 & \dotsb & 0 \\
\vdots & \vdots & \vdots & \vdots & \vdots \\
pz_{n1} & pz_{n2} & pz_{n3} & \dotsb & 1
\end{pmatrix}
\in \bar{N}.
\end{align*}
We say that $f$ is \emph{continuous} if it is continuous as a function on $\ints_p^{n(n-1)/2}$ via $\un{z}\mapsto \bar{N}(\un{z})$. Then $\ssc_s:=\ind_{B(\ints_p)}^{\iw_p,cts}(s)$, where $cts$ stands for \emph{continuous}, is an $A[\iw_p]$-module. If $s^0:(\rats_p^\times)^n\to A^\times$ is the trivial extension of $s$ from $(\ints_p^\times)^n$ to $(\rats_p^\times)^n$ (that is, we set $s^0(d)=1$ for any $d\in(\rats_p^\times)^n$ whose entries are powers of $p$), $\ssc_s$ is isomorphic to $\ind_{B(\rats_p)}^{B(\rats_p)\iw_p,cts}(s^0)$ by restriction of functions from $B(\rats_p)\iw_p$ to $\iw_p$. Consequently it has an action by $B(\rats_p)\iw_p$.

It will be useful to write out the natural action of $\iw_p$ on $f\in\ssc_s$ more explicitly in terms of the coordinates $z_{ij}$. To do this, we interpret them as Pl\"ucker coordinates on $\bar{N}(\un{z})$. Recall that for any $1\le j\le n$ and subset $\sg$ of $\{1,\dotsc,n\}$ with $\#\sg=j$, the Pl\"ucker coordinate $Z_{j,\sg}$ associated to $(j,\sg)$ is the algebraic function on $GL_n$ given by the determinant of the minor associated to the rows corresponding to $\sg$ and the first $j$ columns. 

Give $\rats_p^n$ the standard basis $e_1,\dotsc,e_n$ and interpret elements of $\rats_p^n$ as \emph{horizontal} vectors. Give $\w^j(\rats_p^n)$ the corresponding standard basis
\[
\{e_\sg=e_{k_1}\w\dotsb\w e_{k_j}\mid \sg=\{k_1<\dotsb<k_j\}\subset\{1,\dotsc,n\}\},
\]
ordered lexicographically, and again interpret elements of $\w^j(\rats_p^n)$ as horizontal vectors. Let $1_j=\{1,\dotsc,j\}$. If $GL_n(\rats_p)$ acts on $\rats_p^n$ by right multiplication of horizontal vectors (the transpose of the standard action), and $\it_j:GL_n(\rats_p)\inj GL(\w^j(\rats_p^n))$ gives the induced action of $GL_n(\rats_p)$ on $\w^j(\rats_p^n)$ (where again $GL(\w^j(\rats_p^n))$ acts on $\w^j(\rats_p^n)$ by right multiplication of horizontal vectors), then for $x\in GL_n(\rats_p)$, $Z_{j,\sg}(x)$ is the coefficient of $e_{1_j}$ in $e_\sg\cdot\it_j(x)$, or the entry of $\it_j(x)$ in the $\sg$th row and first column. If $b=(b_{ij})\in B(\rats_p)$, $\it_j(b)$ is also upper triangular, so the coefficient of $e_{1_j}$ in $e_\sg\cdot\it_j(xb)=e_\sg\cdot\it_j(x)\cdot\it_j(b)$  is $Z_{j,\sg}(x)$ times the top left entry of $\it_j(b)$, which is $b_{11}\dotsb b_{jj}=:t_j(b)$. That is, we have
\[
Z_{j,\sg}(xb)=t_j(b)Z_{j,\sg}(x).
\]
So $Z_{j,\sg}$ is invariant under right multiplication by $N$, and $Z_{j,\sg/1}:=Z_{j,\sg}/Z_{j,1_j}$ is invariant under right multiplication by $B$. Let
\[
Z_{j,\sg}(u^{-1}x)=\sum_{\#\tau=j}a_{j,\sg,\tau}(u)Z_{j,\tau}(x)
\]
(note that $a_{j,\sg,\tau}(u)\in\ints_p$, and if $u\in\iw_p$ then $a_{j,1_j,1_j}(u)\in\ints_p^\times$) so that
\[
Z_{j,\sg/1}(u^{-1}x)=\frac{a_{j,\sg,1_j}+\sum_{\#\tau=j,\tau\neq 1_j}a_{j,\sg,\tau}(u)Z_{j,\tau/1}(x)}{a_{j,1_j,1_j}+\sum_{\#\tau=j,\tau\neq 1_j}a_{j,1_j,\tau}(u)Z_{j,\tau/1}(x)}.
\]
For $i\ge j$, let $\sg_{ij}=\{1,\dotsc,j-1,i\}$ (so $\sg_{jj}=1_j$); then we can see that
\[
Z_{j,\sg_{ij}}(\bar{N}(\un{z}))=\begin{cases}
Z_{j,1_j}(\bar{N}(\un{z}))=1 & \text{ if } i=j \\
pz_{ij} & \text{ if } i>j.
\end{cases}
\]
Thus $z_{ij}$, or technically $pz_{ij}$, is indeed a Pl\"ucker coordinate for $\bar{N}(\un{z})$ when $i>j$. Now using the Iwahori decomposition for $\iw_p$, let 
\[
u^{-1}\bar{N}(\un{z})=\bar{N}(\un{uz})T(u,\un{z})N(u,\un{z})
\]
where $T(u,\un{z})\in T$ and $N(u,\un{z})\in N$. So if $f\in\ssc_s$, we have
\begin{align*}
(uf)(\bar{N}(\un{z}))=f(u^{-1}\bar{N}(\un{z})) &=f(\bar{N}(\un{uz})T(u,\un{z})N(u,\un{z})) \\
&=s(T(u,\un{z}))f(\bar{N}(\un{uz})).
\end{align*} 
We wish to write $\un{uz}$ and $T(u,\un{z})$ in terms of $u$ and $\un{z}$. But we have
\begin{align*}
Z_{j,\sg_{ij}}(u^{-1}\bar{N}(\un{z})) &=Z_{j,\sg_{ij}}(\bar{N}(\un{uz})T(u,\un{z})N(u,\un{z})) \\
&=Z_{j,\sg_{ij}}(\bar{N}(\un{uz}))t_j(T(u,\un{z})).
\end{align*}
So in fact, setting $i=j$, we find
\[
t_j(T(u,\un{z}))=Z_{j,\sg_{jj}}(u^{-1}\bar{N}(\un{z}))=\sum_{\#\tau=j}a_{j,1_j,\tau}(u)Z_{j,\tau}(\bar{N}(\un{z}))
\]
where $Z_{j,\tau}(\bar{N}(\un{z}))$ is by definition a polynomial in the variables $\{z_{kl}\}_{l\le j,k>l}$ with coefficients in $p\ints_p$. Then when $i>j$, we have
\[
Z_{j,\sg_{ij}}(u^{-1}\bar{N}(\un{z}))=Z_{j,\sg_{ij}}(\bar{N}(\un{uz}))t_j(T(u,\un{z}))=p(uz)_{ij}Z_{j,\sg_{jj}}(u^{-1}\bar{N}(\un{z}))
\]
so
\[
p(uz)_{ij}=Z_{j,\sg_{ij}/1}(u^{-1}\bar{N}(\un{z}))=\frac{a_{j,\sg,1_j}+\sum_{\#\tau=j,\tau\neq 1_j}a_{j,\sg,\tau}(u)Z_{j,\tau/1}(\bar{N}(\un{z}))}{a_{j,1_j,1_j}+\sum_{\#\tau=j,\tau\neq 1_j}a_{j,1_j,\tau}(u)Z_{j,\tau/1}(\bar{N}(\un{z}))}
\]
where $Z_{j,\tau/1}(\bar{N}(\un{z}))$ is again a polynomial in the variables $\{z_{kl}\}_{l\le j,k>l}$ with coefficients in $p\ints_p$.


\subsection{Notation for subgroups of $\iw_p$}
\label{subgpnotation}

Since we will work with numerous subgroups of $\iw_p$, we will introduce some notation to identify them. If $\un{c}=(c_{ij})\in\ints_{\ge0}^{n\times n}$ is any $n\times n$ matrix of nonnegative integers, we will write
\[
\Gam(\un{c})=\{(x_{ij})\in GL_n(\ints_p)\mid p^{c_{ij}}\mid (x_{ij}-\del_{ij})\text{ for all }i,j\}
\]
(where $\del_{ij}$ is $1$ if $i=j$ and $0$ otherwise). One can compute that $\Gam(\un{c})$ is a group precisely when $c_{ij}\le c_{ik}+c_{kj}$ for all $i,j,k$. Note that this means that if $\Gam(\un{c})$ is a group, then so is $T(\ints_p)\Gam(\un{c})$. If we instead only have half a matrix of nonnegative integers $\un{c}=(c_{ij})_{n\ge i>j\ge1}\in\ints_{\ge0}^{n(n-1)/2}$, we will write
\[
\Gam_1(\un{c})=\{(x_{ij})\in\iw_p\mid v(x_{ij})\ge c_{ij}\forall i>j\text{ and }v(x_{ii}-1)\ge \min\{c_{ij}|j<i\}\cup\{c_{ji}|j>i\}\forall i\}
\]
\[
\Gam_0(\un{c})=\{(x_{ij})\in\iw_p\mid v(x_{ij})\ge c_{ij}\forall i>j\}=T(\ints_p)\Gam_1(\un{c})\subset\iw_p.
\]
\begin{defn}
We say that $\un{c}=(c_{ij})_{n\ge i>j\ge1}\in\ints_{\ge0}^{n(n-1)/2}$ is \emph{group-shaped} if $c_{ij}\le c_{ik}+c_{kj}$ for all $k$, where we set $c_{ab}$ to be $0$ if $a\le b$. 
\end{defn}

Thus $\Gam_1(\un{c})$ and $\Gam_0(\un{c})$ are subgroups whenever $\un{c}$ is group-shaped. 

\begin{defn}
We call an $\frac{n(n-1)}{2}$-tuple $\un{c}=(c_{ij})_{n\ge i>j\ge1}\in\ints_{\ge0}^{n(n-1)/2}$ \emph{compatible with} an $n$-tuple $(c_1,\dotsc,c_n)\in\ints_{\ge0}^n$ if $c_i\le \min\{c_{ij}|j<i\}\cup\{c_{ji}|j>i\}$ for all $i$. Equivalently, if we define $\un{c}'\in\ints_{\ge0}^{n\times n}$ by $c_{ij}'=c_{ij}$ for $i>j$, $c_{ii}'=c_i$, and $c_{ij}'=0$ for $i<j$, then $\Gam(\un{c}')$ is a group.
\end{defn}

Then we see that if $\chi=(\chi_1,\dotsc,\chi_n):T(\ints_p)\to\cplx^\times$ is a character of $T(\ints_p)$, and $\un{c}\in\ints_p^{n(n-1)/2}$, $\chi$ extends to a well-defined character of $T(\ints_p)\Gam_1(\un{c})=\Gam_0(\un{c})$, trivial on $\Gam_1(\un{c})$, whenever $\un{c}$ is compatible with $(\cond(\chi_1),\dotsc,\cond(\chi_n))$. 

In the calculations below, whenever we write $\Gam(\un{c})$, $\Gam_0(\un{c})$, or $\Gam_1(\un{c})$ for a matrix or half-matrix of nonnegative integers $\un{c}$, we will implicitly assume that $\un{c}$ has been chosen so that it is in fact a group.

Depending on convenience, we may also overload the above notation in the following ways. First, if $\un{r}=(r_{ij})\in[0,1]^{n\times n}$ is any $n\times n$ matrix of \emph{real numbers} in $[0,1]$, we will write
\[
\Gam(\un{r})=\{(x_{ij})\in GL_n(\ints_p)\mid |x_{ij}-\del_{ij}|\le r_{ij}\text{ for all }i,j\}.
\]
Then $\Gam(\un{r})$ is a group precisely when $r_{ij}\ge r_{ik}r_{kj}$ for all $i,j,k$. We may define $\Gam_1(\un{r}),\Gam_0(\un{r})$ similarly. Second, if $c\in\ints_{>0}$ is a single integer, we will write 
\[
\Gam(c)=\{(x_{ij})\in GL_n(\ints_p)\mid v(x_{ij}-\del_{ij})\ge c\text{ for all }i,j\}.
\]
This is always a group. We may define $\Gam_1(c),\Gam_0(c)$ similarly. Finally, if $r$ is a single real number in $[0,1]$, we will write $\Gam(r),\Gam_1(r),\Gam_0(r)$ for the obvious final abuse of the same notation.

\subsection{The sheaf of $p$-adic automorphic forms on weight space}
\label{sheafpadicautforms}

If $\un{c}=(c_{ij})_{n\ge i>j\ge1}\in\ints_{>0}^{n(n-1)/2}$, we say that $f\in\ssc_s$ is $\un{c}$-locally analytic if, for any $\un{a}=(a_{ij})\in\ints_p^{n(n-1)/2}$, the restriction of $f$ to 
\[
B(\un{a},\un{c})=\{z=(z_{ij})_{n\ge i>j\ge1}\in\ints_p^{n(n-1)/2}\mid z_{ij}\in a_{ij}+p^{c_{ij}}\ints_p\forall i,j\}
\]
is given by a convergent power series in the variables $z_{ij}$ with coefficients in $A$.

\begin{defn}
We call an $\frac{n(n-1)}2$-tuple $\un{c}=(c_{ij})_{n\ge i>j\ge1}\in\ints_{\ge0}^{n(n-1)/2}$ \emph{analytic-shaped} if we have $c_{(j+1)j}=c_{(j+2)j}=\dotsb=c_{nj}$ for all $j$ and $c_{nj}\ge c_{n(j+1)}$ for all $j$. (Note that if $\un{c}$ is analytic-shaped it is also group-shaped.) We call $\un{c}$ \emph{compatible with} an $n$-tuple $(c_1,\dotsc,c_n)\in\ints_{\ge0}^n$ if $c_j\le\min_{l\le j,k>l}c_{kl}$ for all $j$. That is, for each $j_0$, all the entries of $(c_{ij})$ corresponding to matrix entries appearing in or to the left of the $j_0$th column should be at least $c_{j_0}$.
\end{defn}

\begin{defn}
If $\un{c}\in\ints_{>0}^{n(n-1)/2}$ is analytic-shaped, we say that $s:(\ints_p^\times)^n\to A^\times$ is $\un{c}$-locally analytic if there is $(c_1,\dotsc,c_n)$ such that $s$ is $(c_1,\dotsc,c_n)$-locally analytic and $\un{c}$ is compatible with $(c_1,\dotsc,c_n)$.
\end{defn}

\begin{prop}
\label{locanpres}
If $s$ is $(c_1,\dotsc,c_n)$-locally analytic and $f\in\ssc_s$ is $\un{c}$-locally analytic for $\un{c}$ analytic-shaped and compatible with $(c_1,\dotsc,c_n)$ (so that $s$ is $\un{c}$-locally analytic), then $uf$ is also $\un{c}$-locally analytic for all $u\in\iw_p$. 
\end{prop}

\begin{proof}
By the calculations in Section~\ref{plucker}, we have $(uf)(\bar{N}(\un{z}))=s(T(u,\un{z}))f(\bar{N}(\un{uz}))$ where

---$(uz)_{ij}$ is a power series in the variables $\{z_{kl}\}_{l\le j,k>l}$;

---the $j$th diagonal entry of $T(u,\un{z})$, or $\frac{t_j(T(u,\un{z}))}{t_{j-1}(T(u,\un{z}))}$, is also a power series in the variables $\{z_{kl}\}_{l\le j,k>l}$. 

So if we restrict to $\un{z}\in B(\un{a},\un{c})$, the coefficient $(uz)_{ij}$ ranges over a ball of the form $a_{ij}'+p^{\min_{l\le j,k>l}c_{kl}}\ints_p$; since $\un{c}$ is analytic-shaped, we have $c_{ij}\le\min_{l\le j,k>l}c_{kl}$, and we conclude that $\un{uz}$ is also restricted to a ball of the form $B(\un{a}',\un{c})$. Thus $f(\bar{N}(\un{uz}))$ is analytic for $\un{z}\in B(\un{a},\un{c})$. Similarly, $\frac{t_j(T(u,\un{z}))}{t_{j-1}(T(u,\un{z}))}$ ranges over a ball of the form $a_{jj}''+p^{\min_{l\le j,k>l}c_{kl}}\ints_p$; since $c_j\le\min_{l\le j,k>l}c_{kl}$ and $s_j$ is analytic on $a_{jj}'+p^{c_j}\ints_p$, we conclude that $s_j(T(u,\un{z}))$ is analytic for $\un{z}\in B(\un{a},\un{c})$. Thus $(uf)(\bar{N}(\un{z}))$ is analytic for $\un{z}\in B(\un{a},\un{c})$, as desired.
\end{proof}

By Proposition~\ref{locanpres}, if $s$ is $\un{c}$-locally analytic with $\un{c}$ analytic-shaped, the space $\ssc_{s,\un{c}}=\ind_{B(\ints_p)}^{\iw_p,\un{c}-loc.an.}(s)$, where $\un{c}-loc.an.$ stands for $\un{c}$-locally analytic, is well-defined and has an action by $\iw_p$.


We let $\ssc=\ssc_{[\cdot]}=\ind_{B(\ints_p)}^{\iw_p,cts}([\cdot])$. If $\usc_p=\iw_p$, we call
\[
\ssc(G,\usc)=\ind_{B(\ints_p)}^{\iw_p,cts}([\cdot])(G,\usc)
\]
the space of integral $p$-adic automorphic forms for $G$ of level $\usc$; it has an action by $B(\rats_p)\usc$. This gives a sheaf on $\wsc$ whose fiber over $s$ is
\[
\ssc_s(G,\usc)=\ind_{B(\ints_p)}^{\iw_p,cts}(s)(G,\usc).
\]
Similarly, let $\ssc_{W,\un{c}}=\ssc_{[\cdot]_W,\un{c}}=\ind_{B(\ints_p)}^{\iw_p,\un{c}-loc.an.}([\cdot]_W)$ (for any $\un{c}$ such that $[\cdot]_W$ is $\un{c}$-locally analytic). If $\usc_p=\iw_p$, we call
\[
\ssc_{W,\un{c}}(G,\usc)=\ind_{B(\ints_p)}^{\iw_p,\un{c}-loc.an.}([\cdot]_W)(G,\usc)
\]
the space of $\un{c}$-locally analytic $p$-adic automorphic forms for $G$ of level $\usc$; this does \emph{not} have an action by $B(\rats_p)$, as some elements of $B(\rats_p)$ do not preserve the radius of local analyticity, but we will see in the next section that it has an action by a certain submonoid.

\subsection{The operators $U_p^a$}
\label{upa}

If $H$ is any locally compact, totally disconnected topological group, we write $\hsc(H)$ for the $k$-algebra of compactly supported, locally constant $k$-valued functions on $H$ with the convolution product
\[
(\phi_1\star\phi_2)(g)=\int_{h\in H}\phi_1(h)\phi_2(h^{-1}g)d\mu
\]
where $\mu$ is a Haar measure on $H$. This algebra usually has no identity, but many idempotents. If $K$ is a compact open subgroup of $H$, the idempotent $e_K=\frac{\one_K}{\mu(K)}$ projects $\hsc(H)$ onto the subalgebra $\hsc(H\sslash K)$ of functions that are both left- and right- invariant under $K$. If $V$ is a smooth $H$-module, it is an $\hsc(H)$-module via
\[
\phi(v)=\int_H \phi(h)(hv)dh
\]
and similarly $V^K$ is an $\hsc(H\sslash K)$-module.

In the particular case $H=B(\rats_p)\usc$, $V=\ssc_s(G,\usc)$, $K=\usc$, we can rephrase this as follows. We sometimes write $[\usc\zt\usc]$ for the element $\one_{\usc\zt\usc}$ of $\hsc(G(\aff_f)\sslash \usc)$. If $\zt_1,\dotsc,\zt_r$ are left $\usc$-coset representatives of $\usc\zt\usc$, so that
\[
\usc\zt\usc=\coprod_{i=1}^r\zt_i\usc,
\]
then for any $\phi\in \ssc_s(G,\usc)$ and $x\in G(\rats)\bsl G(\aff_f)$, we have
\begin{align*}
[\usc\zt\usc](\phi)(x) &=\int_{G(\aff_f)}[\usc\zt\usc](g)\cdot(g.\phi)(x)dg \\
&=\int_{\usc\zt\usc}g_p\phi(xg)dg
=
\sum_{i=1}^r(\zt_i)_p.\phi(x\zt_i).
\end{align*}
The following is Lemma 4.5.2 of~\cite{chenevier04}, or Proposition 3.3.3 of~\cite{loeffler10}.

\begin{lem}
\label{hcoords}
Fix coset representatives $x_1,\dotsc,x_h$ of $G(\rats)\bsl G(\aff_f)/\usc$, and thus an isomorphism $\ssc_s(G,\usc)\cong\ssc_s^h$. Then we have
\[
[\usc\zt\usc](\phi)(x_j)=\sum_{k=1}^h\sum_{i\mid\zt_i\in x_j^{-1}G(\rats)x_k\usc}(\zt_iu_{ij}^{-1})_p.\phi(x_k)
\]
for some $u_{ij}\in\usc$. That is, the action of $[\usc\zt\usc]$ on $\ssc_s(G,\usc)$ is of the form $\sum T_j\circ\sg_j$, where the $\sg_j$s are compositions of permutation operators on the entries of vectors in $\ssc_s^h$ with projections onto one of the coordinates, and the $T_j$s are diagonal translations of $\ssc_s^h$ by elements of $\usc\zt\usc$.
\end{lem}

\begin{proof}
Write $x_j\zt_i$ in the form $d_{ij}x_{k_{ij}}u_{ij}$ where $d_{ij}\in G(\rats)$ and $u_{ij}\in\usc$. Then
\begin{align*}
[\usc\zt\usc](\phi)(x_j) &=\sum_{i=1}^r(\zt_i)_p.\phi(x_j\zt_i) \\
&=\sum_{i=1}^r(\zt_i)_p.\phi(d_{ij}x_{k_{ij}}u_{ij})
=\sum_{i=1}^r(\zt_iu_{ij}^{-1})_p.\phi(x_{k_{ij}}).
\end{align*}
The values of $i$ for which $k_{ij}=k$ are those for which $\zt_i=x_j^{-1}dx_ku$ for some $d\in G(\rats)$ and $u\in\usc$, that is, $\zt_i\in x_j^{-1}G(\rats)x_k\usc$.
\end{proof}

If $a=(a_1,\dotsc,a_n)\in\ints^n$, we write
\[
u^a=\diag(p^{a_1},\dotsc,p^{a_n})
\]
and define the subgroup
\[
\Sg=\{u^a=\diag(p^{a_1},\dotsc,p^{a_n})\mid a=(a_1,\dotsc,a_n)\in\ints^n\}\subset GL_n(\rats_p)
\]
and its submonoids
\[
\Sg^-=\{u^a=\diag(p^{a_1},\dotsc,p^{a_n})\mid a_1\ge a_2\ge\dotsb\ge a_n\}\subset \Sg
\]
\[
\Sg^{--}=\{u^a=\diag(p^{a_1},\dotsc,p^{a_n})\mid a_1> a_2>\dotsb> a_n\}\subset \Sg^-.
\]
We will frequently choose $\zt$ to be an element of $\Sg^-$. Let
\[
U_p^a=[\usc\diag(p^{a_1},\dotsc,p^{a_n})\usc].
\]

\begin{prop}
\label{scale}
If $f\in\ssc_s$ and $a=(a_1,\dotsc,a_n)\in\ints^n$, $u^a$ acts on $f$ by $z_{ij}\mapsto p^{a_i-a_j}z_{ij}$.
\end{prop}

\begin{proof}
We have
\[
f((u^a)^{-1}\bar{N}(z_{ij}))=f\left(
\begin{pmatrix}
p^{-a_1} & \dotsb & 0 \\
0 & \vdots & 0 \\
0 & \dotsb & p^{-a_n}
\end{pmatrix}
\begin{pmatrix}
1 & 0 & 0 & \dotsb & 0 \\
pz_{21} & 1 & 0 & \dotsb & 0 \\
pz_{31} & pz_{32} & 1 & \dotsb & 0 \\
\vdots & \vdots & \vdots & \vdots & \vdots \\
pz_{n1} & pz_{n2} & pz_{n3} & \dotsb & 1
\end{pmatrix}
\right)
\]
\[
=f\begin{pmatrix}
p^{-a_1} & 0 & 0 & \dotsb & 0 \\
p^{-a_2+1}z_{21} & p^{-a_2} & 0 & \dotsb & 0 \\
p^{-a_3+1}z_{31} & p^{-a_3+1}z_{32} & p^{-a_3} & \dotsb & 0 \\
\vdots & \vdots & \vdots & \vdots & \vdots \\
p^{-a_n+1}z_{n1} & p^{-a_n+1}z_{n2} & p^{-a_n+1}z_{n3} & \dotsb & p^{-a_n}
\end{pmatrix}
\]
\[
=f\left(\begin{pmatrix}
1 & 0 & 0 & \dotsb & 0 \\
p^{a_1-a_2+1}z_{21} & 1 & 0 & \dotsb & 0 \\
p^{a_1-a_3+1}z_{31} & p^{a_2-a_3+1}z_{32} & 1 & \dotsb & 0 \\
\vdots & \vdots & \vdots & \vdots & \vdots \\
p^{a_1-a_n+1}z_{n1} & p^{a_2-a_n+1}z_{n2} & p^{a_3-a_n+1}z_{n3} & \dotsb & 1
\end{pmatrix}
\begin{pmatrix}
p^{-a_1} & \dotsb & 0 \\
0 & \vdots & 0 \\
0 & \dotsb & p^{-a_n}
\end{pmatrix}
\right)
\]
\[
=f(\bar{N}(p^{a_i-a_j}z_{ij}))s^0(u^a)=f(\bar{N}(p^{a_i-a_j}z_{ij})).
\]
\end{proof}

\begin{cor}
If $f\in\ssc_s$ is $\un{c}$-locally analytic and $u^a\in \Sg^-$, then $u^af$ is also $\un{c}$-locally analytic. So translation by $\iw_pu^a\iw_p$ preserves $\ssc_{s,\un{c}}$ (and hence, by Lemma~\ref{hcoords}, $U_p^a$ preserves $\ssc_{s,\un{c}}(G,\usc)$).
\end{cor}

\begin{proof}
When $u^a\in \Sg^-$, we have $a_i-a_j\ge0$ for all $i>j$; thus if $(z_{ij})$ varies in a ball $B(\un{a},\un{c})$, so does $(p^{a_i-a_j}z_{ij})=(u^az_{ij})$.
\end{proof}

Let $\un{c}^0\in\ints_{>0}^{n(n-1)/2}$ be minimal such that $s$ is $\un{c}^0$-locally analytic. 

\begin{cor}
\label{shrink}
If $f\in\ssc_s$ is $\un{c}$-locally analytic and $u^a\in \Sg^{--}$, then $u^af$ is $\un{c}^{--}:=(\max\{c_{ij}-1,c_{ij}^0\})$-locally analytic. So translation by $\iw_pu^a\iw_p$ takes $\ssc_{s,\un{c}}$ into $\ssc_{s,\un{c}^{--}}$ (and hence, by Lemma~\ref{hcoords}, $U_p^a$ takes $\ssc_{s,\un{c}}(G,\usc)$ into $\ssc_{s,\un{c}^{--}}(G,\usc)$).
\end{cor}

\begin{proof}
When $u^a\in \Sg^{--}$, we have $a_i-a_j>0$ for all $i>j$; thus if $(z_{ij})$ varies in a ball $B(\un{a},\un{c})$, then $(p^{a_i-a_j}z_{ij})=(u^az_{ij})$ varies in a smaller ball $B(\un{a}',\un{c+1})$.
\end{proof}

$\ssc_{s,\un{c}}$ is an orthonormalizable $A$-module, for which we choose the following orthonormal basis: for each $\un{a}\in\prod_{n\ge i>j\ge1} \ints_p/p^{c_{ij}}\ints_p$, we choose the set of monomials $\prod_{n\ge i>j\ge1}z_{ij}^{e_{ij}}$ as an orthonormal basis for the restriction of $\ssc_{s,\un{c}}$ to $B(\un{a},\un{c})$; then for $\ssc_{s,\un{c}}$, we may choose as orthonormal basis the set of monomials $\prod_{n\ge i>j\ge1}(z_{ij}^{\un{a}})^{e_{ij}}$, with one copy for each $\un{a}\in\prod_{n\ge i>j\ge1} \ints_p/p^{c_{ij}}\ints_p$.

\begin{cor}
When $a\in \Sg^{--}$, the operator of translation by $u^a$ acts \emph{completely continuously} on $\ssc_{s,\un{c}}$, in the sense that it is a uniform limit of operators with finite-dimensional images. So by Lemma~\ref{hcoords}, $U_p^a$ is completely continuous on $\ssc_{s,\un{c}}(G,\usc)$. 
\end{cor}

\begin{proof}
By Proposition~\ref{scale}, $u^a$ scales $\prod_{n\ge i>j\ge1}(z_{ij}^{\un{a}})^{e_{ij}}$ by $\prod_{n\ge i>j\ge1}p^{(a_i-a_j)e_{ij}}$, which goes to $\infty$ as any $e_{ij}$ goes to $\infty$. Furthermore, since the formulas in Section~\ref{plucker} all have integer coefficients, it is clear that translation by $\iw_p$ is norm $1$.
\end{proof}

Since $U_p^a$ is completely continuous on $\ssc_{s,\un{c}}(G,\usc)$, for any $k$, the matrix of the action of $U_p^a$ (in any basis) has a finite number of nonzero rows mod $p^k$. Suppose that this matrix has $r_k$ rows that are zero mod $p^k$ but nonzero mod $p^{k+1}$. Then for any $N\ge r_0+r_1+\dotsb+r_k$, the coefficient of $X^N$ in the characteristic power series
\[
P_{s,\un{c}}^a(X)=\det(1-XU_p^a|\ssc_{s,\un{c}}(G,\usc))
\]
of $U_p^a$ acting on $\ssc_{s,\un{c}}(G,\usc)$, being a linear combination of minors of size $N\ge r_0+r_1+\dotsb+r_k$, is divisible by $r_1+2r_2+\dotsb+kr_k$. Since this lower bound grows faster than any linear function of $N$, $P_{s,\un{c}}^a(X)$ is an entire function of $X$.

\begin{prop}
\label{radindep}
$P_{s,\un{c}}^a(X)$ is independent of $\un{c}$. (So we will henceforth call it $P_s^a(X)$.)
\end{prop}

\begin{proof}
This follows from applying Corollary 2 of Proposition 7 of~\cite{serre62} to the map $U_p^a:\ssc_{s,\un{c}}(G,\usc)\to\ssc_{s,\un{c}^{--}}(G,\usc)$ from Corollary~\ref{shrink} and the obvious inclusion $\ssc_{s,\un{c}^{--}}(G,\usc)\inj \ssc_{s,\un{c}}(G,\usc)$. 
\end{proof}


Let $U_p^\Sg$ be the subring of $\hsc(G(\aff_f)\sslash \usc)$ generated by the elements $U_p^a$ for $a\in \Sg^-$ and their inverses (which exist, as discussed in Section 6.4.1 of \cite{bc09}). By Proposition 6.4.1 of \cite{bc09}, the map from $k[\Sg]$ to $U_p^\Sg$ sending $u^a$ to $U_p^b(U_p^c)^{-1}$ where $u^b,u^c$ are any elements of $\Sg^-$ such that $u^a=u^b(u^c)^{-1}$ is a well-defined isomorphism of rings. So, in particular, $U_p^\Sg$ is abelian. Let $\hsc$ be a subalgebra of $\hsc(G(\aff_f)\sslash\usc)$ given by the product of $\ints[U_p^\Sg]$ at $p$ and some commutative subalgebra of $\hsc(G(\aff_f^p)\sslash\usc^p)$ away from $p$.

We write $u_i$ for the image of $\diag(1,\dotsc,1,p,1,\dotsc,1)\in k[\Sg]$ in $U_p^\Sg$. If $f$ is an element of an $\hsc$-module $S$ (such as $\ssc_{s,\un{c}}(G,\usc)$) that is a generalized simultaneous eigenvector for $\hsc$, let $u_i(f)=\lam_i f$. We call these the $\lam$-values associated to $f$. We call the subspace generated by all the generalized simultaneous eigenvectors whose associated $\lam$-values are nonzero the finite-slope subspace of $S$, and we denote it by $S^{fs}$.

Unless otherwise specified, we will generally set $\usc$ to be a compact open subgroup of $G(\aff_f)$ given by the product of $\iw_p$ at $p$ and a fixed tame level structure away from $p$ chosen so that $x^{-1}G(\rats)x\cap U_0(p)=1$ for all $x$ (the condition of being ``sufficiently small'' or ``neat'' as described at the end of Section~\ref{padicautforms}). Call this subgroup $U_0(p)$. (Note that for the same reason as in Proposition~\ref{levels} below, our choice of $\iw_p$ as the wild level structure does not actually affect $P_s^a(X)$.)

\subsection{The eigenvariety}
\label{spectralvariety}

Given our setup so far, the eigenvariety is easy to define. For a given $u^a\in\Sg^{--}$, let $\zsc^a$ be the subvariety of $\wsc\times\gra_m$ which, in any subset $W\times\gra_m$ where $W\subset\wsc$ is open affinoid, is cut out by the characteristic power series $P_W^a(X)$ of $U_p^a$ acting on $\ssc_W(G,U_0(p))$. Let $w:\zsc^a\to\wsc$ be the first projection (weight) map, and $a_p^a:\zsc^a\to\gra_m$ the \emph{inverse} of the second projection ($U_p^a$-eigenvalue) map. Then for any point $z\in\zsc^a$, $a_p^a(z)$ is a nonzero eigenvalue of $U_p^a$ acting on $\ssc_{w(z)}(G,U_0(p))$, and for any $w\in\wsc$, all nonzero eigenvalues of $U_p^a$ acting on $\ssc_w(G,U_0(p))$ can be found in the fiber of $\zsc^a$ over $w$. We call $\zsc^a$ the spectral variety associated to $U_p^a$. 

It is convenient to fix a particular choice of $u^a\in\Sg^{--}$; we will choose $a=(n-1,n-2,\dotsc,1,0)$. From now on, we will write $U_p=U_p^{(n-1,n-2,\dotsc,1,0)}$ and $\zsc=\zsc^{(n-1,n-2,\dotsc,1,0)}$. We call an eigenform $f\in\ssc_w(G,U_0(p))$ finite-slope if $U_pf\neq0$ (i.e. the valuation, or slope. of the $U_p$-eigenvalue is finite, and $f$ appears on the eigenvariety), and infinite-slope otherwise. 


Since $\hsc$ is commutative, we can construct the space $\dsc$ whose points correspond to systems of eigenvalues of all Hecke operators in $\hsc$, including in particular all $U_p^a$s simultaneously, by simply taking $\dsc$ to be the finite cover of $\zsc$ which, over an affinoid $W\subset\wsc$, is given by the MaxSpec of the image of $\hsc\ten\Lam^n$ in the endomorphism ring of $\ssc_W(G,U_0(p))$. Then $\dsc$ inherits the weight map $w:\zsc^a\to\wsc$ and each eigenvalue map $a_p^a:\dsc\to\gra^m$. Note that $\dsc\to\zsc^a$ is degree $1$ away from multiple roots of $P_W^a(X)$, hence degree $1$ away from a Zariski-closed subset of $\wsc$ of lower dimension. So in general, the bounds and geometric properties we get for $\zsc^a$ will also apply to $\dsc$. For most of this paper, we will focus on the properties of $\zsc$ and/or $\zsc^a$ for any fixed $a$. 

For additional details on properties of $\zsc^a$ and $\dsc$ and their proofs, see~\cite{chenevier04} or~\cite{buzzard07}.




\section{Locally algebraic weights}
\label{localgweights}

In this section, we analyze classical automorphic forms of locally algebraic weights and their associated automorphic representations. In Section~\ref{localgrepdef}, we define these spaces of classical forms and check their basic properties, including that they embed into the infinite-dimensional spaces of Section~\ref{padicautforms}. In Section~\ref{classicality}, we reproduce Bella\"iche-Chenevier's slope criterion guaranteeing that a given form is classical, phrased to work for locally algebraic weights instead of just algebraic weights; while this is not directly needed for our purposes, it is useful to give a sense of where classical forms fit in among the world of all $p$-adic automorphic forms. In Section~\ref{assocautrep}, we explain the standard translation between classical forms and automorphic representations. In Section~\ref{smstruct}, we analyze certain Iwahori subrepresentations that may appear in the local component at $p$ of such an automorphic representation, including a particularly important irreducible subrepresentation. In Section~\ref{rochecalc}, we apply the work of Roche to a calculation of Hecke eigenvalues in ramified principal series. In Section~\ref{repstructure}, we identify a subspace of forms whose associated automorphic representations have ramified principal series as their local components at $p$, and compute their $U_p$-eigenvalues in terms of the parameters of the corresponding principal series.

\subsection{$p$-adic automorphic forms of locally algebraic weights}
\label{localgrepdef}

In Section~\ref{padicautforms}, we defined classical forms of algebraic weights via the algebraic representation $S_t(k)$ of $GL_n(\rats_p)$. This construction may be generalized to locally algebraic weights as follows. Let $\chi=\chi_1\dotsb\chi_n$ be a finite character of $(\ints_p^\times)^n$. Then $t\chi$ is a locally algebraic character of $(\ints_p^\times)^n$, in the sense that it is algebraic upon restriction to $\prod_{i=1}^n(a_i+p^{c_i}\ints_p)$ for some choice of $c_i$s and any nonzero $a_i$s. Similarly to earlier notation, for a positive integer $c$, let
\[
B(\un{a},c)=\{z=(z_{ij})_{n\ge i>j\ge1}\in\ints_p^{n(n-1)/2}\mid z_{ij}\in a_{ij}+p^{c}\ints_p\forall i,j\}
\] 
Then there are two equivalent definitions of the space
\[
S_{t\chi,c}:=\ind_{B(\ints_p)}^{\iw_p,c-loc. alg.}(t\chi)
\]
where $c-loc.alg.$ stands for \emph{$c$-locally algebraic}. The first is through the usual induction operator above, as follows. We say that $f\in\ind_{B(\ints_p)}^{\iw_p}(t\chi)$ is \emph{$c$-locally algebraic} if it has an algebraic extension to $B(\un{a},c)$ for all $\un{a}\in\ints_p^{n(n-1)/2}$ of degree bounded as follows: writing $f$ as a polynomial in the variables $Z_{i,k/1}$ as in Section~\ref{plucker}, we require that for each fixed $i$, the degree of $f$ as a polynomial in all the variables $Z_{i,k/1}$ should be at most $t_i-t_{i+1}=:m_i$. As in Proposition~\ref{locanpres}, one can see using the formulas in Section~\ref{plucker} that assuming $\cond(\chi_i)\le c$ for all $i$, this condition is invariant under right translation by $\iw_p$. 


The second definition, coming from the perspective of Loeffler (Section 2.5 of~\cite{loeffler10}), is
\[
\left(\ind_{B(\ints_p)}^{\iw_p,alg}t\right)\ten\left(\ind_{B(\ints_p)/B(\ints_p)\cap \Gam(c)}^{\iw_p/\Gam(c)}\chi\right).
\]
Note that $\Gam(c)$ is normal in $\iw_p$ because it is the kernel of the reduction map from $\iw_p$ to the corresponding group with coefficients in $\ints_p/p^c\ints_p$.

Except for an annoying technical distinction which we will discuss at the end of this subsection, the space $\ind_{B(\ints_p)}^{\iw_p,alg}t$ is the same (as an $\iw_p$-representation) as the space $S_t(k)$ defined in Section~\ref{padicautforms}, since $\iw_p$ is Zariski-dense in $GL_n$. Let $d_t=\dim\ind_{B(\ints_p)}^{\iw_p,alg}t$. We will now check that the two definitions just given are actually equivalent. 

\begin{prop}
The natural map 
\[
\left(\ind_{B(\ints_p)}^{\iw_p,alg.}t\right)\ten\left(\ind_{B(\ints_p)/B(\ints_p)\cap \Gam(c)}^{\iw_p/\Gam(c)}\chi\right)\to \ind_B^{\iw_p,c-loc. alg.}(t\chi)
\]
\[
f\ten g\mapsto fg
\]
is an isomorphism.
\end{prop}

\begin{proof}
To construct an inverse, let $\phi\in \ind_{B(\ints_p)}^{\iw_p,c-loc.alg.}(t\chi)$. Let $\phi_{alg}:\iw_p\to\cplx$ be defined by
\[
\phi_{alg}(b\bar{n})=t(b)\phi'(\bar{n})
\]
for all $b\in B,\bar{n}\in\bar{N}\cap \iw_p$, where $\phi'$ is the unique algebraic extension of $\phi|_{\bar{N}\cap \Gam(c)}$ to $\bar{N}\cap \iw_p$. Let $\phi_{sm}:\iw_p/\iw_p\cap \Gam(c)\to\cplx$ be defined by
\[
\phi_{sm}(\bar{b}\bar{\bar{n}})=\chi(b)(\phi/\phi')(\bar{n})
\]
where $b,\bar{n}$ are any lifts of $\bar{b}\in B/B\cap \Gam(c)$, $\bar{\bar{n}}\in(\bar{N}\cap \iw_p)/(\bar{N}\cap \Gam(c))$. Then we have $\phi_{alg}\ten\phi_{sm}\mapsto\phi$, which suffices to prove surjectivity. 

Injectivity follows from dimension counting: both sides have dimension $d_tp^{c\binom n2}$.
\end{proof}

\begin{rem}
There is a simple isomorphism of $\iw_p$-representations
\[
\ind_{B(\ints_p)/B(\ints_p)\cap \Gam(c)}^{\iw_p/\Gam(c)}\chi\isom\ind_{\Gam_0(c)}^{\iw_p}\chi
\]
so we could just as easily have phrased this section in terms of $\ind_{\Gam_0(c)}^{\iw_p}\chi$. For now, we have no particular reason to do this, but it may be more convenient for future work.
\end{rem}

We call
\[
S_{t\chi,c}(G,\usc)=\ind_B^{\iw_p,c-loc. alg.}(t\chi)(G,\usc)
\]
the space of classical $p$-adic automorphic forms on $G$ of weight $t\chi$, radius $c$, and level $\usc$. By the definitions, it embeds into $\ssc_{t\chi}(G,\usc)$, and we call its image a classical subspace of $\ssc_{t\chi}(G,\usc)$. The following proposition is a quick generalization of part 4 of Lemma 4 of~\cite{buzzard04}.


\begin{prop}
\label{levels}
For any positive integers $c$, $d$, and $e$ with $d\le e$ and $c+d-e\ge1$, we have a natural vector space isomorphism
\[
S_{t\chi,c}(G,\usc^p\Gam_0(d))\cong S_{t\chi,c+d-e}(G,\usc^p\Gam_0(e))
\]
such that systems of $\hsc$-eigenvalues on the left (where $\hsc$ is obtained with respect to $\usc^p\Gam_0(d)$) go to identical systems of $\hsc$-eigenvalues on the right (where $\hsc$ is obtained with respect to $\usc^p\Gam_0(e)$).
\end{prop}

\begin{proof}
For the purposes of this proposition, let $X=G(\rats)\bsl G(\aff_f)$. The left-hand side is the subset of
\begin{equation}
\label{proplevelshom}
(\hom_{set}(X,\cplx_p)\ten_{\cplx_p}\ind_B^{\iw_p,c-loc. alg.}(t\chi))^{\Gam_0(e)}
\end{equation}
that remains invariant under a set of coset representatives $A$ for $\Gam_0(e)\bsl\Gam_0(d)$. This subset has a map by restriction of the second factor to 
\[
(\hom(X,\cplx_p)\ten \osc)^{\Gam_0(e)}
\]
where $\osc$ is the space of functions on $B((p^{e-d}\ints_p)^{n(n-1)/2},c)$ that are algebraic on each ball $B(\un{a},c)$. The map is an isomorphism: if $\phi\in (\hom(X,\cplx_p)\ten \osc)^{\Gam_0(e)}$, its inverse $\psi$ may be defined by 
\[
\psi(x)(z)=\phi(xa^{-1})(\bar{N}^{-1}(\bar{N}(z)a^{-1}))\text{ for }a\in A\text{ such that }za^{-1}\in B\left((p^{e-d}\ints_p)^{n(n-1)/2},c\right).
\]
In $\bar{N}(z)a^{-1}$, $a$ should be interpreted as a coset representative for $\Gam_0(e-d+1)\bsl\iw_p$. Note that this inverse depends on the choice of coset representatives $A$. Now $B((p^{e-d}\ints_p)^{n(n-1)/2},c)$ is isomorphic to $B(\ints_p^{n(n-1)/2},c+d-e)$ via multiplication by $p^{d-e}$, so $(\hom(X,\cplx_p)\ten \osc)^{\Gam_0(e)}$ is the desired right-hand side.

To check that the Hecke operator action is preserved, it suffices to note that the Hecke operator action on the left-hand side can be calculated on its image in (\ref{proplevelshom}). 
\end{proof}


\begin{cor}
\label{emb1}
For all positive integers $c$ and group-like $\un{d}\in\ints_{\ge0}^{n(n-1)/2}$, we have a vector space embedding 
\[
S_{t\chi,c}(G,\usc^p\Gam_0(\un{d}))\inj \ssc_{t\chi}(G,U_0(p))
\]
preserving systems of $\hsc$-eigenvalues.
\end{cor}

\begin{proof}
Let $d=\max d_{ij}$. Then we have an embedding 
\[
S_{t\chi,c}(G,\usc^p\Gam_0(\un{d}))\inj S_{t\chi,c}(G,\usc^p\Gam_0(d)).
\]
By Proposition~\ref{levels}, we have an isomorphism
\[
S_{t\chi,c}(G,\usc^p\Gam_0(d))\cong S_{t\chi,c+d-1}(G,\usc^p\Gam_0(1))=S_{t\chi,c+d-1}(G,\usc^p\iw_p).
\]
The space on the right certainly embeds into $\ssc_{t\chi}(G,\usc^p\iw_p)=\ssc_{t\chi}(G,U_0(p))$ as discussed above.
\end{proof}

For future reference, it will be important to note the following distinction between the space $S_{t,1}(G,U_0(p))$ defined above and the space $S_t(k)(G,U_0(p))$ of classical algebraic automorphic forms defined in Section~\ref{padicautforms}, which is that they are identical except for the normalization of the action of the $U_p$-operator. This is because, as in the beginning of Section~\ref{plucker}, the action of $u^a$ on $S_t=\ind_{B(\ints_p)}^{\iw_p,alg}t=\ind_{B(\rats_p)}^{B(\rats_p)\iw_p,alg}t^0$ implicitly arises from the extension of $t$ to $t^0:(\rats_p^\times)^n\to\cplx$ where $t^0(u^a)=1$, whereas the action of $u^a$ on $S_t(k)$ arises from the algebraic character $t:(\rats_p^\times)^n\to\cplx$, for which we can compute $t(u^a)=p^{\sum_i a_it_i}$. Thus we have
\[
U_p^a|S_{t,1}(G,U_0(p))=p^{\sum_i a_it_i}U_p^a|S_t(k)(G,U_0(p)).
\]

\subsection{A classicality theorem following Bella\"iche-Chenevier}
\label{classicality}

This is essentially Proposition 7.3.5 of \cite{bc09}. We will just summarize the proof with modifications so that it also works for locally algebraic weights.

\begin{thm}
Let $f\in\ssc_{t\chi}(G,\usc)$ where $t\chi=(t_1\chi_1,\dotsc,$ $t_n\chi_n)$, in which the $t_i$ are integers such that $t_1\ge\dotsb\ge t_n$ and the $\chi_i$ are finite, such that $f$ is an eigenform for all operators $U_p^{(a_1,\dotsc,a_n)}$. Let $\lam_1,\dotsc,\lam_{n-1}$ be the $\lam$-values associated to $f$ as defined at the end of Section~\ref{upa}. If
\[
v(\lam_1\lam_2\dotsb\lam_i)<t_i-t_{i+1}+1
\]
for all $i=1,\dotsc,n-1$, then $f$ is classical (i.e. lies in the image of $S_{t\chi,c}(G,\usc)$ for any $c$ such that this is well-defined).
\end{thm}

\begin{proof}
Let $V=\rats_pv\dsm\rats_pR$ be a finite-dimensional vector space generated by a nonzero vector $v$ and a lattice $R$. In some basis whose first vector is $v$, we define the following matrix groups, where a lowercase letter refers to a single matrix entry, an uppercase letter refers to a larger submatrix of the appropriate size, a number refers to a submatrix consisting of copies of that number of the appropriate size, and the subscript after the matrix refers to its coefficient ring:
\[
H=GL_{\rats_p}(V),\hspace{0.1in} P=\mattwo{a}{B}{0}{D}_{\rats_p}\subset H,
\]
\[
\bar{N}=\mattwo{1}{0}{C}{I_R}_{\ints_p},\hspace{0.1in} J=\mattwo{a}{B}{pC}{D}_{\ints_p}\subset GL_{\ints_p}(V)
\]
so that
\[
J=(\bar{N}\cap J)\times(P\cap J)=\mattwo{1}{0}{pC}{I_R}_{\ints_p}\times\mattwo{a}{B}{0}{D}_{\ints_p}.
\]
We have an isomorphism $\al:\bar{N}\cap J\to R$ given by
\[
\al\mattwo{1}{0}{pC}{I_R}=C.
\]
We also define
\[
\Ufr^-=\mattwo{p^k}{0}{0}{p^{\le k}GL_{\ints_p}(R)}_{\ints_p},\Ufr^{--}=\mattwo{p^k}{0}{0}{p^{\le k-1}GL_{\ints_p}(R)}_{\ints_p}.
\]
We define $\Mfr=\inn{\Ufr^-,J}$. Let $\chi:P\to\rats_p^\times$ be the character of $P$ acting on $\rats_p v$. We have a $\cplx_p[H]$-equivariant isomorphism
\[
\sym^m(V\ten_{\rats_p}\cplx_p)^\vee\to \ind_P^{H,alg}(\chi^m)
\]
\[
\phi\mapsto(h\mapsto\phi(h(e))).
\]
We get a natural $\Mfr$-equivariant map
\[
\ind_P^{H,alg}(\chi^m)\to\ind_P^{JP,an}(\chi^m)
\]
by restriction. Let $\del:\Mfr\to\cplx_p^\times$ be the character such that $\del(J)=1$ and $\del(u)=p^a$ if
\[
u=\mattwo{p^a}{0}{0}{U}\in \Ufr^-.
\]
Let $e_1,\dotsc,e_n$ be the standard basis of $\rats_p^n$. Let $V_i=\w^i(\rats_p^n)$, $v_i=e_1\w e_1\w\dotsb \w e_i$, $m_i=t_i-t_{i+1}$ if $i<n$ and $m_n=t_n$, and $R_i$ be the $\ints_p$-span of the elements $e_{j_1}\w\dotsb e_{j_i}$ with $j_1<\dotsb<j_i$ and $(j_1,\dotsc,j_i)\neq(1,\dotsc,i)$. Then for $i=1,\dotsc,n$, we get
\[
H_i,P_i,\chi_i,\bar{N}_i,J_i,\al_i,\Ufr_i^-,\Ufr_i^{--},\Mfr_i,\del_i
\]
as defined above.

Write $S_i(\cplx_p)^\vee$ for the space $\ind_{P_i}^{H_i,alg}(\chi_i^{m_i})$ viewed as a representation of $G(\rats_p)$ via $\w^i:G(\rats_p)\to H_i$. Write $\ssc_i(m_i)$ for the space $\ind_{P_i}^{H_i,an}(\chi_i^{m_i})\ten\del_i^{m_i}$ viewed as a representation of $\Mfr$ via $\w^i$. We have surjections
\[
\bigotimes_{i=1}^mS_i(\cplx_p)^\vee\to S_t(\cplx_p)^\vee
\]
\[
\widehat{\bigotimes}_{i=1}^m \ssc_i(m_i)\to \ssc_t
\]
which are equivariant with respect to $G(\rats_p)$ and $\Mfr$ respectively, both given by the formula
\[
(f_1,\dotsc,f_m)\mapsto\left(g\mapsto\prod_{i=1}^m f_i(\w^i(g))\right).
\]
Let
\[
Q=\ssc_t/(S_t(\cplx_p)^\vee\ten\del_t)\ten\left(\ind_{B(\ints_p)/B(\ints_p)\cap \Gam(c)}^{\iw_p/\iw_p\cap \Gam(c)}\chi\right),
\]
\[
Q'=\widehat{\bigotimes}_{i=1}^m \ssc_i(m_i) / \left(\bigotimes_{i=1}^mS_i(\cplx_p)^\vee\ten\del_i^{m_i}\right)\ten\left(\ind_{B(\ints_p)/B(\ints_p)\cap \Gam(c)}^{\iw_p/\iw_p\cap \Gam(c)}\chi\right),
\]
\[
Q_i'=\left(\widehat{\bigotimes}_{j\neq i}\ssc_j(m_j)\right)\ten\left(\ssc_i(m_i)/S_i(\cplx_p)^\vee\ten\del_i^{m_i}\right)\ten\left(\ind_{B(\ints_p)/B(\ints_p)\cap \Gam(c)}^{\iw_p/\iw_p\cap \Gam(c)}\chi\right).
\]
Then we have a surjection 
\[
Q'(G,\usc)\surj Q(G,\usc)
\]
and an injection 
\[
Q'(G,\usc)\inj\prod_{i=1}^n Q_i'(G,\usc). 
\]
We wish to show that if $w\in Q(G,\usc)$ satisfies the hypotheses of the theorem, that is, $u_i(w)=\lam_i w$ with the $\lam_i$ satisfying the given inequalities, then $w=0$. We can instead check this claim for $w'\in Q'(G,\usc)$ satisfying the same condition, and for this it suffices to check that the image $w_i'$ of $w'$ vanishes in $Q_i'(G,\usc)$ for each $i$. Let $U_i=\frac{\prod_{j=1}^iu_i}{p^{m_i+1}}$, so that $w_i'$ has $U_i$-eigenvalue $\frac{\lam_1\lam_2\dotsb\lam_i}{p^{m_i+1}}w$, which has norm $>1$. Thus it suffices to check that $U_i$ has norm $\le1$ on $Q_i'(G,\usc)$, which follows from the claim that any element of the form
\[
\frac{g\left(\prod_{j=1}^iu_i\right)g'}{p^{m_i+1}}
\]
for $g,g'\in\iw_p$ has norm $\le 1$ on $Q_i'$. This follows from Lemma 7.3.6 of \cite{bc09}.
\end{proof}

\subsection{Automorphic representations associated to automorphic forms of locally algebraic weights}
\label{assocautrep}

Fix an isomorphism $\it_p:\bar{\rats}_p\isom\cplx$. Let $f\in\ssc_{t\chi}(G,U_0(p))$ be a $p$-adic automorphic form coming from some classical subspace $S_{t\chi,c}(G,U_0(p))$. Let $W=\ind_{B(\ints_p)}^{\iw_p,c-loc. alg.}(t\chi)$, so that $f$ is a function $G(\rats)\bsl G(\aff_f)\to W$. Following the proof of Proposition 3.8.1 of \cite{loeffler10}, let $W=W^{sm,c}(\chi)\ten S_t(\cplx)$, where, as in Section~\ref{localgrepdef},
\[
W^{sm,c}(\chi)=\ind_{B(\ints_p)/B(\ints_p)\cap \Gam(c)}^{\iw_p/\iw_p\cap \Gam(c)}\chi,
\]
\[
S_t(\cplx)=\ind_{B(\ints_p)}^{\iw_p,alg}t,
\]
and let $\rho_{sm},\rho_{alg}$ denote the actions of $\iw_p$ on $W^{sm,c}(\chi)\ten S_t(\cplx)$ given by acting on only the first factor and only the second factor respectively. Then we can define a function $f_\infty:G(\aff)\to W$ by $f_\infty(g)=\rho_{alg}(g_\infty^{-1}\it_p(g_p))f(g_f)$ which satisfies the relation
\[
f_\infty(gu)=\rho_{sm}(u_p)^{-1}\rho_{alg}(u_\infty)^{-1}f_\infty(g)
\]
for all $u\in G(\real)U_0(p)$. Equivalently, $f_\infty$ can be viewed as the function
\[
f_\infty^\vee:(W^{sm,c}(\chi)^\vee\ten S_t(\cplx)^\vee)\times G(\rats)\bsl G(\aff)\to \cplx
\]
\[
(\phi,x)\mapsto \phi(f_\infty(x))
\]
which satisfies
\[
f_\infty^\vee(\phi,xu)=\phi(f_\infty(xu))=\phi(\rho_{sm}(u_p)^{-1}f_\infty(x))=f_\infty^\vee(u_p\phi,x)
\]
for all $u\in U_0(p)$. Thus for each $\phi\in W^{sm,c}(\chi)^\vee\ten S_t(\cplx)^\vee$, the function $f_\infty^\vee(\phi,\cdot)$ is an element of $C(G(\rats)\bsl G(\aff),\cplx)$ which generates under right translation by $\iw_p$ a representation containing an irreducible component of $W^{sm,c}(\chi)^\vee$ . The right translates of $f_\infty^\vee(\phi,\cdot)$ under $G(\aff)$ generate an automorphic representation $\pi_f$ of $G(\aff)$ which decomposes as a tensor product $\bigotimes_p'\pi_{f,p}$. We are interested in describing the structure of $\pi_{f,p}$.

Note that this process is reversible, in that given $\psi\in C(G(\rats)\bsl G(\aff),\cplx)$ which generates a representation containing an irreducible component of $W^{sm,c}(\chi)^\vee$ under right translation by $\iw_p$, we get a unique $f_\psi\in S_{t\chi,c}(G,U_0(p))$.


\subsection{Structure of $W^{sm,c}(\chi)$}
\label{smstruct}

We are interested in the representation
\[
W^{sm,c}(\chi)=\ind_{B(\ints_p)/B(\ints_p)\cap \Gam(c)}^{\iw_p/\Gam(c)}\chi
\]
of $\iw_p$. Note that there is an obvious embedding $W^{sm,c}(\chi)\inj W^{sm,c+1}(\chi)$ which takes $f\in W^{sm,c}(\chi)$ to the composition of $f$ with the reduction map $\iw_p/\Gam(c+1)\to\iw_p/\Gam(c)$. 

Let $J$ be the compact open subgroup of $GL_n(\rats_p)$ corresponding to $\chi$ defined in Section 3 of \cite{roche98}; we have $J=\Gam(\un{c})$ where 
\[
c_{ij}=\begin{cases}
0 & \text{ if }i=j \\
\left\lfloor\frac{\cond(\chi_i\chi_j^{-1})}2\right\rfloor & \text{ if } i<j\\
\left\lfloor\frac{\cond(\chi_i\chi_j^{-1})+1}2\right\rfloor  & \text{ if } i>j.
\end{cases}
\]
Then $\chi$ extends to a character of $J$ which we will also call $\chi$; it is defined by the equation $\chi(j^-jj^+)=\chi(j)$ when $j^-\in J\cap \bar{N}(\ints_p)$, $j\in T(\ints_p)$, and $j^+\in J\cap N(\ints_p)$. Let $U^{sm}(\chi):=\ind_J^{\iw_p}\chi$.

Now note that $W^{sm,c}(\chi)$ contains the vector
\[
f(\bar{x})=\begin{cases}
\chi(j)\chi(b) & \text{ if }\bar{x}=\bar{j}\bar{b}\in\iw_p/\Gam(c)\text{ with }j\in J\text{ and }b\in B(\ints_p),\\
0 & \text{otherwise.}
\end{cases}
\]
Note furthermore that for any $j\in J$ and $\bar{x}\in \iw_p/\Gam(c)$, we have
\[
(jf)(\bar{x})=f(j^{-1}\bar{x})=\chi(j^{-1})f(\bar{x})=\chi^{-1}(j)f(\bar{x})
\]
so that $f$ is $(J,\chi^{-1})$-isotypic.

\begin{prop}
\label{irred}
Assume $\chi=(\chi_1,\dotsc,\chi_n)$ satisfies

\begin{enumerate}
\item \label{jbig} for all $i\neq j$, $\cond(\chi_i\chi_j^{-1})=\max(\cond(\chi_i),\cond(\chi_j))$; and

\item \label{charclose} for all $i\neq j$ with $i,j\neq n$, $\cond(\chi_i)<2\cond(\chi_j)$.
\end{enumerate}

Then $U^{sm}(\chi)$ is irreducible.
\end{prop}

\begin{proof}
By Mackey's criterion, it is necessary and sufficient to show that for any $s\in\iw_p\setminus J$, the characters $\chi$ and $\chi^s:j\mapsto \chi(sjs^{-1})$ are not identically equal on $J\cap s^{-1}Js$. If $s\in\iw_p\setminus J$, let $t=s^{-1}$. Then there is a pair $i\neq j$ such that $t_{ji}$ is not divisible by $p^{c_{ji}}$. Among all such $i\neq j$, choose a pair such that either

---among the integers $c_{kj}+c_{jk}-v(t_{jk})$, $1\le k\le n$, $k\neq j$, $c_{ij}+c_{ji}-v(t_{ji})$ is the unique maximal one;

---or, if this is not possible, among the integers $c_{kj}+c_{jk}-v(t_{jk})$, $1\le k\le n$, $k\neq j$, $c_{ij}+c_{ji}-v(t_{ji})$ is maximal and $i$ is minimal such that this is the case.

Let $x\in J$ be the identity except for the $ij$th entry; let $x_{ij}=b$. Note that we must have $p^{c_{ij}}|b$. We will show that we can choose $b$ such that $sxt\in J$ and $1=\chi(x)\neq\chi(sxt)$, and hence $\chi(x)\neq\chi^s(x)$, as desired. 

The matrix $xt$ is the same as $t$ except for the $i$th row, which is
\[
(t_{i1}+bt_{j1},\dotsc,t_{in}+bt_{jn}).
\]
The $kk$th entry of $sxt$ is 
\[
s_{k1}t_{1k}+\dotsb+s_{ki}(t_{ik}+bt_{jk})+\dotsb+s_{kn}t_{nk}=s_{k1}t_{1k}+\dotsb+s_{kn}t_{nk}+bs_{ki}t_{jk}=1+bs_{ki}t_{jk}.
\]
Because of condition~\ref{jbig}, one can check that for all $j\in J$, we have $\chi(j)=\chi_1(j_{11})\cdot\dotsb\cdot\chi_n(j_{nn})$. 
So we wish to choose $b$ such that
\[
\chi_1(1+bs_{1i}t_{j1})\dotsb\chi_i(1+bs_{ii}t_{ji})\dotsb\chi_j(1+bs_{ji}t_{jj})\dotsb\chi_n(1+bs_{ni}t_{jn})\neq1.
\]
Note that for all $k\neq i,j$, we have 
\[
v(s_{ki})+c_{ij}+c_{ji}-v(t_{ji})>c_{jk}+c_{kj}-v(t_{jk}).
\] 
This is just because we chose $i,j$ such that $c_{ij}+c_{ji}-v(t_{ji})\ge c_{jk}+c_{kj}-v(t_{jk})$, and such that if equality holds then $k>i$, in which case $v(s_{ki})\ge1$ since $s\in\iw_p$. So if we choose $b$ such that $v(b)=c_{ij}+c_{ji}-v(t_{ji})-1\ge c_{ij}$, then we have 
\[
v(bs_{ki}t_{jk})\ge c_{jk}+c_{kj}=\cond(\chi_k\chi_j^{-1})
\] 
for all $k\neq i,j$, hence
%
$
\chi_k(1+bs_{ki}t_{jk})=1.
$
Then we have
\[
\chi_1(1+bs_{1i}t_{j1})\dotsb\chi_i(1+bs_{ii}t_{ji})\dotsb\chi_j(1+bs_{ji}t_{jj})\dotsb\chi_n(1+bs_{ni}t_{jn})
\]
\[
=\chi_i(1+bs_{ii}t_{ji})\chi_j(1+bs_{ji}t_{jj})=\chi_i(1+bs_{ii}t_{ji})\chi_j\left(1+b\left(\sum_{k\neq i}s_{ki}t_{jk}\right)\right)
\]
since $v(bs_{ki}t_{jk})\ge\cond(\chi_j)$, but this is 
\[
\chi_i(1+bs_{ii}t_{ji})\chi_j(1-bs_{ii}t_{ji})
\]
since $\sum_{k}s_{ki}t_{jk}=\sum_k t_{jk}s_{ki}=(ts)_{ji}=0$, and this can be rewritten as
\[
\frac{\chi_i}{\chi_j}(1+bs_{ii}t_{ji})\chi_j(1-b^2s_{ii}^2t_{ji}^2)=\frac{\chi_i}{\chi_j}(1+bs_{ii}t_{ji})
\]
because if $i>j$ then $v(b^2)\ge 2c_{ij}\ge c_{ij}+c_{ji}$ and if $i<j$ then $v(b^2)\ge 2c_{ij}\ge c_{ij}+c_{ji}-1$ and $v(t_{ji})\ge1$. But since $v(b)<c_{ij}+c_{ji}-v(t_{ji})$, we have $v(bs_{ii}t_{ji})<\cond(\chi_i\chi_j^{-1})$, so we can choose $b$ to make $\frac{\chi_i}{\chi_j}(1+bs_{ii}t_{ji})\neq1$.

Finally, we verify that for this choice of $b$, we actually have $sxt\in J$. The $kl$th entry of $sbt$ is
\[
s_{k1}t_{1l}+\dotsb+s_{ki}(t_{il}+bt_{jl})+\dotsb+s_{kn}t_{nk}=\del_{kl}+bs_{ki}t_{jl}.
\]
We have
\[
v(bs_{ki}t_{jl})=c_{ij}+c_{ji}-v(t_{ji})-1+v(s_{ki})+v(t_{jl})
\]
\[
=c_{ij}+c_{ji}-v(t_{ji})-(c_{lj}+c_{jl}-v(t_{jl}))+c_{lj}+c_{jl}-1+v(s_{ki})
\]
\[
\ge c_{lj}+c_{jl}-1+v(s_{ki})\ge c_{kl}
\]
by condition~\ref{charclose}.
\end{proof}

\begin{rem}
We do not believe that either condition~\ref{jbig} or condition~\ref{charclose} of Proposition~\ref{irred} should be strictly necessary. Notably, most of the proof of Proposition~\ref{irred} can be easily rephrased to avoid references to condition~\ref{jbig}. Our only sticking point is the calculation of $\chi(j)$ in terms of $j_{11},\dotsc,j_{nn}$.
\end{rem}

We call $\chi$ ``simple'' if it satisfies the conditions of Proposition~\ref{irred}. By Frobenius reciprocity, we conclude that if $\chi$ is simple, $W^{sm,c}(\chi)$ contains $U^{sm}(\chi)$. Also note that if $\chi$ satisfies condition~\ref{jbig}, and the conductors of the nontrivial components of $\chi=(\chi_1,\dotsc,\chi_{n-1},1)$ are, in order from least to greatest, $c_{(1)}\le c_{(2)}\le\dotsb\le c_{(n-1)}$, then the index of $J$ in $\iw_p$ is
\[
p^{c_{(1)}+2c_{(2)}+\dotsb+(n-1)c_{(n-1)}-\frac{n(n-1)}2}=:p^{j(\chi)},
\]
and this is $\rank(U^{sm}(\chi))$. So if $\chi$ is simple and $\cond(\chi_i)=c$ for all $i\neq n$, then $W^{sm,c}(\chi)$ and $U^{sm}(\chi)$ have the same dimension and must actually be isomorphic.

\subsection{Hecke eigenvalues of ramified principal series}
\label{rochecalc}

The representations we are interested in will turn out to be ramified principal series of $GL_n(\rats_p)$, so we now cover the properties of these that we will need. To harmonize with the literature, for this section only, we will use different conventions from the rest of the paper. If $\chi=(\chi_1,\dotsc,\chi_n):(\rats_p^\times)^n\to\cplx$ is a smooth character of $T(\rats_p)=(\rats_p^\times)^n$, we will write
\[
i_{B(\rats_p)}^{GL_n(\rats_p)}\chi=\{f:GL_n(\rats_p)\to\cplx\mid f(bg)=\chi(b)f(g)\text{ for all }g\in GL_n(\rats_p)\text{ and }b\in B(\rats_p)\}
\]
for the representation of $GL_n(\rats_p)$ with the given underlying vector space and the right translation action of $GL_n(\rats_p)$. We let $\del^{1/2}:(\rats_p^\times)^n\to\cplx$ be the modulus character
\[
\del^{1/2}:=(|\cdot|^{(n-1)/2},|\cdot|^{(n-3)/2},\dotsc,|\cdot|^{(1-n)/2}).
\]
Then we define
\[
\pi(\chi):=\pi(\chi_1,\dotsc,\chi_n):=i_{B(\rats_p)}^{GL_n(\rats_p)}(\chi\del^{1/2}).
\]
The representation $\pi(\chi)$ is called the normalized parabolic induction of $\chi$.
Assume that for all $i\neq j$, we have $\chi_i(p)\neq \chi_j(p)p$. Let $J=\Gam(\un{c})$ be the subgroup defined at the beginning of Section~\ref{smstruct}. Let $\hsc(GL_n(\rats_p)\sslash J,\chi)$ be the subspace of $\hsc(GL_n(\rats_p))$ generated by the functions $\phi:GL_n(\rats_p)\to\cplx$ satisfying $\phi(j_1xj_2)=\chi(j_1)^{-1}\phi(x)\chi(j_2)^{-1}$ for all $j_1,j_2\in J$ and $x\in GL_n(\rats_p)$. 

\begin{lem}
The $(J,\chi)$-isotypic piece of $\pi(\chi)$ is $1$-dimensional.
\end{lem}

\begin{proof}
By Theorem 6.3 of \cite{roche98}, $\hsc(GL_n(\rats_p)\sslash J,\chi)$ is abelian. (To be precise, the theorem gives an isomorphism between $\hsc(GL_n(\rats_p)\sslash J,\chi)$ and $\hsc(W_\chi^0,S_\chi^0)\ten\cplx[\Om_\chi]$, where by our assumption that $\chi_i(p)\neq \chi_j(p)p$ for $i\neq j$, we have $W_\chi^0=S_\chi^0=1$, $\hsc(W_\chi^0,S_\chi^0)=\cplx$, and $\Om_\chi=\ints^n$.) Thus the $(J,\chi)$-isotypic piece of $\pi(\chi)$ decomposes as a representation of $\hsc(GL_n(\rats_p)\sslash J,\chi)$ into $1$-dimensional pieces. But by Theorem 9.2 of \cite{roche98}, because $\pi(\chi)$ is irreducible, the $(J,\chi)$-isotypic piece of $\pi(\chi)$ is irreducible as a representation of $\hsc(GL_n(\rats_p)\sslash J,\chi)$. Thus it is itself $1$-dimensional.
\end{proof}

\begin{lem}
\label{psevals}
If $a=(a_1,\dotsc,a_n)$ is such that $a_1\ge a_2\ge\dotsb\ge a_n$, the action of the element $[Ju^aJ]$ of $\hsc(GL_n(\rats_p)\sslash J,\chi)$ corresponding to $u^a=\diag(p^{a_1},\dotsc,p^{a_n})$ on the $(J,\chi)$-isotypic piece of $\pi(\chi)$ is multiplication by
\[
\chi(u^a)=\chi_1(p^{a_1})\dotsb\chi_n(p^{a_n}).
\]
\end{lem}

\begin{proof}
The $(J,\chi)$-isotypic piece is generated by
\[
f(g)=\begin{cases}
(\chi\del)(b)\chi(j) & \text{ if }g=bj\text{ with }b\in B(\rats_p)\text{ and }j\in J,\\
0 & \text{otherwise.}
\end{cases}
\]
This is just because this function $f$ satisfies the $(J,\chi)$-isotypic condition by construction, and is well-defined because $(\chi\del)(b)=\chi(b)$ for any $b\in B\cap J$. We claim that
\begin{equation}
\label{lempsevalsnewvec}
f(ju^a)=\chi(j)\chi(u^a)\del(u^a)=\chi(j)\chi(u^a)\del(u^a)f(1)\text{ for any }j\in J.
\end{equation}
The lemma follows from this, because if $Ju^aJ=\coprod_{i=1}^rj_iu^aJ$, then
\[
([Ju^aJ]f)(1)=\sum_{i=1}^r\chi(j_i)^{-1}f(j_iu^a)=\sum_{i=1}^r\chi(u^a)\del(u^a)f(1)=\chi(u^a)f(1)
\]
because $r=\del(u^a)^{-1}$ (since the same calculation as in Proposition~\ref{scale} shows that the index of $J$ in $[(u^a)^{-1}Ju^a]J$ is
\[
p^{\sum_{i<j}(a_i-a_j)}=p^{(n-1)a_1+(n-3)a_2+\dotsb+(1-n)a_n}).
\]
To prove (\ref{lempsevalsnewvec}), first write $j=j^-j^0j^+$ where $j^-\in J\cap\bar{N}(\ints_p)$, $j^0\in T(\ints_p)$, and $j^+\in J\cap N(\ints_p)$. Then we have $\chi(j)=\chi(j^0)$. Let $j_1^+=j^0j^+(j^0)^{-1}$; then $j_1^+\in J\cap N(\ints_p)$ as well, and $j=j^-j_1^+j^0$. Use Lemmas 3.1 and 3.2 of~\cite{roche98} to write $j^-j_1^+=j_2^+j_2^-c$, where $j_2^+\in J\cap N(\ints_p)$, $j_2^-\in J\cap\bar{N}(\ints_p)$, and $c\in T(\ints_p)$ is a correction torus element in the kernel of $\chi$. Then we have
\[
ju^a=j_2^+j_2^-cj^0u^a=u^a[(u^a)^{-1}j_2^+u^a][(u^a)^{-1}j_2^-u^a](j^0c).
\]
We have $(u^a)^{-1}j_2^+u^a\in N(\rats_p)$, and by the same calculation as in Proposition~\ref{scale}, we have $(u^a)^{-1}j_2^-u^a\in J\cap\bar{N}(\ints_p)$. Therefore
\[
f(ju^a)=f(u^a[(u^a)^{-1}j_2^+u^a][(u^a)^{-1}j_2^-u^a](j^0c))
\]
\[
=(\chi\del)(u^a[(u^a)^{-1}j_2^+u^a])\chi([(u^a)^{-1}j_2^-u^a](j^0c))=(\chi\del)(u^a)\chi(j^0)=(\chi\del)(u^a)\chi(j)
\]
as desired.
\end{proof}

\subsection{Structure of automorphic representations of locally algebraic weights}
\label{repstructure}


Let $f\in S_{t\chi,c}(G,U_0(p))$ be a classical eigenform, and let $\pi_{f,p}$ be an irreducible component of the local component at $p$ of the automorphic representation $\pi_f$ associated to $f$ in Section~\ref{assocautrep}. We first verify a standard fundamental fact about the structure of $\pi_{f,p}$ for those $f$ associated to points on the eigenvariety $\dsc$. 

\begin{prop}
\label{jacquet}
$f$ may be associated to a classical point $x$ on $\dsc$ (equivalently, $f$ is finite-slope) if and only if $\pi_{f,p}$ has nonzero Jacquet module with respect to $B$, or equivalently is a subquotient of a principal series.
\end{prop}

To show this, we use the following proposition of Casselman in \cite{casselman95} on canonical liftings. Recall the submonoid $\Sg^{--}$ from Section~\ref{upa}. Also recall that by definition, the Jacquet module of a representation $(\pi,V)$ of $GL_n(\rats_p)$ with respect to a parabolic subgroup $P$ with Levi factorization $P=MN$ is the space $V_N$ of $N$-coinvariants of $V$, which is naturally a representation of $M$. (See Sections 3 and 4 of~\cite{casselman95} for more basic information about Jacquet modules.)

\begin{prop}[Casselman, Proposition 4.1.4]
\label{canonlift}
Let $(\pi,V)$ be an admissible representation of $GL_n(\rats_p)$, $P=MN$ a parabolic subgroup with Levi factorization, and $K_0 =\bar{N}_0 M_0 N_0$ a compact open subgroup with Iwahori factorization. If $u^a\in \Sg^{--}$, then the projection from $V^{K_0}$ to $V_N^{M_0}$ given by $[K_0u^aK_0]$ is a surjection. If $u^aN_1(u^a)^{-1}\subseteq N_0$, where $N_1$ is a compact subgroup of $N$ such that $V^{K_0}\cap V(N)\subseteq V(N_1)$, then the projection is an isomorphism.
\end{prop}

\begin{proof}[Proof of Proposition~\ref{jacquet}]
We apply Proposition~\ref{canonlift} with $P=B$ and $M=T$. 

Suppose first that $\pi_{f,p}=(\pi,V)$ has nonzero Jacquet module. Let $v\in V_N$ be a nonzero vector and let $M_0$ be a compact open subgroup of $M$ fixing $v$. Let $K_0$ be a compact open subgroup of $G$ such that $K_0\cap M=M_0$. By the proposition of Casselman, $[K_0u^aK_0]V$ surjects onto $V_N^{M_0}\neq0$, so is itself nonzero. Thus $[K_0u^aK_0]$ has some nonzero eigenvalue, corresponding to an eigenvector in $\pi_{f,p}$ which must be the image of an eigenform in $S_{t\chi,c}(G,U_0(p))$ by the procedure of Section~\ref{assocautrep}.

Now suppose in the other direction that $\pi=\pi_{f_x}$, so contains a vector $\im(f_x)$ with nonzero Hecke eigenvalue for $[K_0u^aK_0]$ for some compact open subgroup $K_0$ and all $a\in \Sg^-$. Choose a compact subgroup $N_1$ of $N$ such that $V^{K_0}\cap V(N)\subseteq V(N_1)$. We claim that for sufficiently large powers $(u^a)^k$ of $u^a$, we must have $(u^a)^k N_1(u^a)^{-k}\subseteq N_0$; this is just the effect of conjugation by $(u^a)^k$ on the $ij$th entry of $N_1$ is scaling by $p^{k(a_i-a_j)}$, and $k(a_i-a_j)$ becomes arbitrarily large as $k$ does. Then $[K_0(u^a)^kK_0]V\cong V_N^{M_0}$, and we must have $V_N^{M_0}\neq 0$.
\end{proof}

Now assume $\chi_n$ is trivial and let $c_0=\max_{1\le i<n}\cond(\chi_i)$. Let $(W^{sm,c_0}(\chi))^\perp$ be the complement of $W^{sm,c_0}(\chi)$ in $W^{sm,c}(\chi)$. We now observe that as $c$ goes to infinity, almost all eigenforms in $S_{t\chi,c}(G,U_0(p))$ are infinite-slope.

\begin{prop}
\label{notfs}
Suppose that $f$ is an eigenform in $((W^{sm,c_0}(\chi))^\perp\ten S_t)(G,U_0(p))\subset S_{t\chi,c}(G,U_0(p))$. Then $U_pf=0$, $f$ is not associated to a point on the eigenvariety, and $\pi_{f,p}$ is not a subquotient of a principal series.
\end{prop}


\begin{proof}
By Proposition~\ref{radindep}, in order for $U_pf$ to be nonzero, $f$ must lie in $\ssc_{t\chi,c_0}(G,U_0(p))$. But the intersection of $\ssc_{t\chi,c_0}(G,U_0(p))$ with $((W^{sm,c_0}(\chi))^\perp\ten S_t)(G,U_0(p))$ is trivial.
\end{proof}

Now recall the $\iw_p$-representation $U^{sm}(\chi):=\ind_J^{\iw_p}\chi$ from the beginning of Section~\ref{smstruct}. By the discussion at the end of Section~\ref{smstruct}, if $\chi$ is simple, $(U^{sm}(\chi)\ten S_t)(G,U_0(p))$ is a subspace of $S_{t\chi,c}(G,U_0(p))$. Even if $\chi$ is not simple, the following is true. 

\begin{prop}
\label{fs-appear}
There is a vector space embedding of $(U^{sm}(\chi)\ten S_t)(G,U_0(p))^{fs}$ (as at the end of Section~\ref{upa}) into $\ssc_{t\chi,c}(G,U_0(p))$ which preserves systems of $\hsc$-eigenvalues.
\end{prop}

\begin{proof}
Let $\tilde{H}_{la}^0$ be the space, as in Definition 3.2.3 of~\cite{emerton06} and the discussion surrounding it, of continuous $\bar{\rats}_p$-valued functions on $G(\rats)\bsl G(\aff)$ that are locally constant on cosets of $G(\aff_f^p)$ and locally analytic on cosets of $G(\rats_p)$. Let $J_B$ be Emerton's locally analytic Jacquet module functor as constructed in~\cite{emerton06jac}. Finally, let $e^p$ be the idempotent of $\hsc(G(\aff))$ away from $p$ corresponding to the tame part of $U_0(p)$. According to Proposition 3.10.3 of~\cite{loeffler10}, $\ssc_{t\chi,c}(G,U_0(p))^{fs}$ is isomorphic as an $\hsc$-module to
\[
e^p\left(J_B\left(\tilde{H}_{la}^0\right)\ten_\rats (t\chi)\right)^{T(\ints_p)}
\]
where we write $(t\chi)$ for the representation of $T(\ints_p)$ given by the character $t\chi$. (Note that Loeffer uses different conventions from us, hence writes $J_{\bar{B}}$ instead of $J_B$.) For $f\in (U^{sm}(\chi)\ten S_t)(G,U_0(p))$, define $f_\infty^\vee:(U^{sm}(\chi)\ten S_t)^\vee\times G(\rats)\bsl G(\aff)\to\cplx$ by the same construction as in Section~\ref{assocautrep}, with $U^{sm}(\chi)$ in place of $W^{sm,c}(\chi)$. Then if $\phi\in (U^{sm}(\chi)\ten S_t)^\vee$ is the vector taking an element of $U^{sm}(\chi)\ten S_t$ to its evaluation on $\id\in\iw_p$, then $f_\infty^\vee(\phi,\cdot)$ is a continuous $\cplx_p$-valued function on $G(\rats)\bsl G(\aff)$ that is locally constant on cosets of $G(\aff_f^p)$ and analytic on cosets of the compact open subgroup $J$ of $G(\rats_p)$. Thus we get an inclusion
\begin{align*}
(U^{sm}(\chi)\ten S_t)(G,U_0(p)) &\inj \tilde{H}_{la}^0 \\
f &\mapsto f_\infty^\vee(\phi,\cdot).
\end{align*}
Let $\pi$ be the $G(\rats_p)$-subrepresentation of $\tilde{H}_{la}^0$ generated by $\im((U^{sm}(\chi)\ten S_t)(G,U_0(p)))$. By Proposition~\ref{canonlift}, there is $a\in\Sg^{--}$ such that $U_p^a$ gives an isomorphism $(U^{sm}(\chi)\ten S_t)(G,U_0(p))^{fs}\isom J_B(\pi)$. Since $T(\ints_p)$ acts the same way on $J_B(\pi)$ and on $t\chi$, these identifications combine to an inclusion
 \[
(U^{sm}(\chi)\ten S_t)(G,U_0(p))^{fs}\isom J_B(\pi)\inj e^p\left(J_B\left(\tilde{H}_{la}^0\right)\ten_\rats (t\chi)\right)^{T(\ints_p)}\isom \ssc_{t\chi,c}(G,U_0(p))^{fs}.
\]
\end{proof}


Now it turns out that in fact all of $(U^{sm}(\chi)\ten S_t)(G,U_0(p))$ is finite-slope.

\begin{prop}
\label{manyps}
Suppose that $f$ is an eigenform in $(U^{sm}(\chi)\ten S_t)(G,U_0(p))$. Then $\pi_{f,p}$ is a subquotient of a principal series, in particular one of the form $\pi(\psi_1,\dotsc,\psi_n)$ where $\psi_i:\rats_p^\times\to\cplx$ are characters of $\rats_p^\times$ whose restrictions to $\ints_p^\times$ are the same as $\chi_1,\dotsc,\chi_n$ in some order.
\end{prop}

\begin{proof}
We know that $\pi_{f,p}$ is an irreducible subrepresentation of $C(GL_n(\rats_p),\cplx)$ whose restriction to $\Gamma_0(p)$ admits a nontrivial homomorphism from $(U^{sm}(\chi))^\vee$. By Frobenius reciprocity, the restriction of $\pi_{f,p}$ to $J$ admits a nontrivial homomorphism from the representation of $J$ given by $\chi$; that is, it contains a $(J,\chi)$-isotypic vector. By Theorem 7.7 of~\cite{roche98}, $\pi_{f,p}$ is a subquotient of $\pi(\psi_1,\dotsc,\psi_n)$. 
\end{proof}

\begin{rem}
If $\chi$ is simple, one can also prove Proposition~\ref{manyps} by noting that if $\pi_{f,p}$ admits a nontrivial homomorphism from the irreducible $(U^{sm}(\chi))^\vee$, it must in fact contain all of $(U^{sm}(\chi))^\vee$, in particular the $(J,\chi)$-isotypic vector. We are grateful to Jessica Fintzen for pointing out the more general proof above.
\end{rem}


By Propositions~\ref{manyps} and~\ref{jacquet}, $U_p$ is injective on the space $(U^{sm}(\chi)\ten S_t)(G,U_0(p))$. Furthermore, for an eigenform $f$ in this space, it is possible to compute the eigenvalues of $U_p$ from the structure of $\pi_{f,p}$ or vice versa, as follows. From now on, for convenience, we will sometimes refer to the algebraic weight $(t_1,\dotsc,t_{n-1},0)$, $t_1\ge\dotsb\ge t_{n-1}$, by its successive differences $m_1=t_1-t_2$, $m_2=t_2-t_3$, ..., $m_{n-1}=t_{n-1}$.

\begin{prop}
\label{lamtopsi}
Suppose that $\chi_i(p)\neq\chi_j(p)p$ for all $i\neq j$, and $f$ is an eigenform in $(U^{sm}(\chi)\ten S_t)(G,U_0(p))$. Suppose that we have $\pi_{f,p}=\pi(\psi_1,\dotsc,\psi_n)$ (note that this is an equality because for such $\chi$, $\pi(\psi_1,\dotsc,\psi_n)$ is irreducible). The $\lam$-values associated to $x$ as in Section~\ref{upa} satisfy
\[
\lam_i=p^{(n-1)/2-i+1-m_n-m_{n-1}-\dotsb-m_{n-i+1}}\psi_i(p).
\]
\end{prop}

\begin{proof}
We are given that for all $u^a\in\Sg^-$, we have $U_p^af=\lam_1^{a_1}\dotsb\lam_n^{a_n}f$. Since any eigenvector of $U_p^a=[U_0(p)u^aU_0(p)]$ is also an eigenvector of $[Ju^aJ]$, we can calculate its eigenvalue using $[Ju^aJ]$ instead. Let
\[
Ju^aJ=\coprod_{i=1}^r\zt_iJ.
\]
Then for any $\phi\in U^{sm}(\chi)\ten S_t$, we have
\[
(U_p^af)_\infty^\vee(\phi,x)=\phi(\rho_{alg}(x_\infty^{-1}\it_p(x_p))(U_p^af)(x_f))
\]
\[
=\del^{1/2}(u^a)p^{-\sum a_it_i}\phi\left(\rho_{alg}(x_\infty^{-1}\it_p(x_p(\zt_i)_p))\sum_{i=1}^r\rho_{sm}((\zt_i)_p)f(x\zt_i)\right).
\]
Choose $\phi=\phi_{sm}\ten\phi_{alg}$ so that $\phi_{sm}$ is a $(J,\chi)$-isotypic vector in $U^{sm}(\chi)$. Then by definition
\[
\phi(\rho_{sm}((\zt_i)_p)f(x\zt_i))=\psi((\zt_i)_p)\phi(f(x\zt_i))
\]
so
\[
(U_p^af)_\infty^\vee(\phi,x)=\del^{1/2}(u^a)p^{-\sum a_it_i}\sum_{i=1}^r\psi((\zt_i)_p)\phi\left(\rho_{alg}(x_\infty^{-1}\it_p(x_p(\zt_i)_p))f(x\zt_i)\right)
\]
\[
=\del^{1/2}(u^a)p^{-\sum a_it_i}\sum_{i=1}^r\psi((\zt_i)_p) f_\infty^\vee(\phi,x\zt_i).
\]
That is, we have
\[
\sum_{i=1}^r\psi((\zt_i)_p) f_\infty^\vee(\phi,x\zt_i)=\del^{-1/2}(u^a)p^{\sum a_it_i}\lam_1^{a_1}\dotsb\lam_n^{a_n}f_\infty^\vee(\phi,x).
\]
So the image of $f_\infty^\vee(\phi,\cdot)$ in $\pi_{f,p}$ is a $J$-new vector ($\hsc(J,\psi)$-module). By Lemma~\ref{psevals}, we have
\[
\del^{-1/2}(u^a)p^{\sum a_it_i}\lam_1^{a_1}\dotsb\lam_n^{a_n}=\psi_1(p^{a_1})\dotsb\psi_n(p^{a_n}).
\]
The proposition follows.
\end{proof}

In summary, we have found a finite-slope subspace of $\ssc_{t\chi,c}(G,U_0(p))$ of rank
\[
\rank((U^{sm}(\chi)\ten S_t)(G,U_0(p)))=hd_tp^{j(\chi)}
\]
where, as before, $d_t=\dim \ind_{B(\ints_p)}^{\iw_p,alg.}t$, $h=\#(G(\rats)\bsl G(\aff_f)/U_0(p))$, and $j(\chi)$ is defined as at the end of Section~\ref{smstruct}. If $\chi$ is simple, this subspace is contained in the classical space $S_{t\chi,c}(G,U_0(p))$, and furthermore, we can extend Proposition~\ref{notfs} to show that it accounts for all the finite-slope forms in $S_{t\chi,c}(G,U_0(p))$.

\begin{prop}
\label{notfs-v2}
Suppose that $\chi$ is simple and $f$ is an eigenform in $(U^{sm}(\chi)^\perp\ten S_t)(G,U_0(p))\subset S_{t\chi,c}(G,U_0(p))$. Then $U_pf=0$, $f$ is not associated to a point on the eigenvariety, and $\pi_{f,p}$ is not a subquotient of a principal series.
\end{prop}

\begin{proof}
Let $c_i=\cond(\chi_i)$, and first assume that $c_1\ge\dotsb\ge c_{n-1}$. Then the tuple $\un{c}^0\in\ints_{>0}^{n(n-1)/2}$ associated to $\chi$ defined immediately before Corollary~\ref{shrink} satisfies $c_{ij}^0=c_i$ for all $i>j$. We claim that the intersection of $\ssc_{t\chi,\un{c}^0}(G,U_0(p))$ with $S_{t\chi,c}(G,U_0(p))$ is precisely $(U^{sm}(\chi)\ten S_t)(G,U_0(p))$. 

To show that $(U^{sm}(\chi)\ten S_t)(G,U_0(p))$ is contained in $\ssc_{t\chi,\un{c}^0}(G,U_0(p))$, it suffices to note that $U^{sm}(\chi)\ten S_t$ is contained in $\ssc_{t\chi,\un{c}^0}$, which is true because $f\ten\phi\in U^{sm}(\chi)\ten S_t$ is clearly contained in $\ssc_{t\chi,\un{c}^0}$ for the vector $f\in U^{sm}(\chi)$ defined at the beginning of Section~\ref{smstruct} and any $\phi\in S_t$, and $U^{sm}(\chi)\ten S_t$ is irreducible.

To show that $(U^{sm}(\chi)\ten S_t)(G,U_0(p))$ exhausts $\ssc_{t\chi,\un{c}^0}(G,U_0(p)) \cap S_{t\chi,c}(G,U_0(p))$, we simply note that the latter space also has dimension $hd_tp^{j(\chi)}$, since as a vector space it is $h$ copies of the locally algebraic vector subspace of $\ssc_{t\chi,\un{c}^0}$. By Proposition~\ref{radindep}, in order for $U_pf$ to be nonzero, $f$ must lie in $\ssc_{t\chi,\un{c}^0}(G,U_0(p))$; this completes the proof.

If the $c_i$ are not in decreasing order, by the beginning of Section~\ref{myubdproof}, the finite-slope subspace of $S_{t\chi,c}(G,U_0(p))$ has the same dimension as that of $S_{t\chi^w,c}(G,U_0(p))$ where $\chi^w$ is $\chi$ with the components rearranged so that the $c_i$ are in decreasing order. This completes the argument for all $\chi$ simple.
\end{proof}

The combination of Propositions~\ref{manyps} and~\ref{notfs-v2} gives us the following precise version of Theorem~\ref{classifyps-vague}.

\begin{thm}
\label{classifyps}
If $\chi$ is simple, then the finite-slope classical subspace of $\ssc_{t\chi,c}(G,U_0(p))$ is precisely $(U^{sm}(\chi)\ten S_t)(G,U_0(p))$.
\end{thm}

\section{Bounds on the Newton polygon}
\label{boundproofs}

In this section, we prove Theorem~\ref{mybounds}. We prove Part~\ref{mylbd} in Section~\ref{mylbdproof} and Part~\ref{myubd} in Section~\ref{myubdproof}. In Section~\ref{combine}, we prove a modified version of Part~\ref{myubd} which generates infinitely many upper bound points for the same Newton polygon. 

Fix a character of $\Del^n$, and thus a particular polydisc in $\wsc^n$. Over the subset of this polydisc where $T_n=0$, we have 
\[
\det(I-XU_p)=\sum_{N\ge0}c_N(T_1,\dotsc,T_{n-1})X^N\in\ints_p\ps{T_1,\dotsc,T_{n-1}}\ps{X}
\]
with $c_0(T_1,\dotsc,T_{n-1})=1$. 

\subsection{A lower bound on the Newton polygon}
\label{mylbdproof}

The following is Part~\ref{mylbd} of Theorem~\ref{mybounds}.

\begin{thm}
There are constants $A_1$, $C$ (depending on $n$, $p$, and $h$) such that for all $T_1,\dotsc,T_{n-1}$ such that all $|T_j|>\frac1p$, the Newton polygon of $\sum_{N\ge0}c_N(T_1,\dotsc,T_{n-1})X^N$ lies above the points
\[
\left(x, \left(A_1x^{1+\frac2{n(n-1)}}-C\right)\cdot\min v(T_j)\right)
\]
for all $x$.
\end{thm}


\begin{proof}
We use the language of~\cite{jn16}. Fix an index $a$, and restrict to the subset $|T_a|\ge|T_j|$ for all $j\neq a$. Let 
$R^\circ$ be the $T_a$-adic completion of
\[
\ints_p\ps{T_1,\dotsc,T_{n-1}}\left[\frac{p}{T_a},\frac{T_1}{T_a},\dotsc,\frac{T_{n-1}}{T_a}\right]
\]
and let $R=R^\circ[1/T_a]$. Give $R$ the norm $|r|=\inf\{p^{-n}\mid r\in T_a^n R_\eta^\circ\}$. Let $[\cdot]_R:(\ints_p^\times)^n\to R^\times$ be the universal character with values in $R$. Let $\dcal$ be the continuous $R$-dual of $\ind_{B(\ints_p)}^{\iw_p,cts}[\cdot]_R$. 

$\dcal$ is orthonormalizable with the following norm: choose topological generators $\bar{n}=(\bar{n}_1,\dotsc,\bar{n}_{n(n-1)/2})$ 
for $\bar{N}$, for example the matrix coefficients $pz_{21},pz_{31},pz_{32},\dotsc,pz_{n(n-1)}$ of Section~\ref{plucker}. Let $\bar{\nfr}_i\in\dcal$ be the Dirac distribution at $\bar{n}_i$ on $\bar{N}$. For $\eta=(\eta_1,\dotsc,\eta_{n(n-1)/2})\in\ints_{\ge0}^{n(n-1)/2}$, write $\bar{\nfr}^{\eta}:=\prod_{i=1}^{n(n-1)/2}\bar{\nfr}_i^{\eta_i}$ and $|\eta|=\sum_{i=1}^{n(n-1)/2}\eta_i$ for short. Then $\{\bar{\nfr}^\eta\}_{\eta\in\ints^{n(n-1)/2}}$ is a basis for $\dcal$, and the norm is
\[
\left\|\sum_\eta d_\eta\bar{\nfr}^\eta\right\|_r=\sup_\eta|d_\eta|r^{|\eta|}.
\]
Let $\dcal^r$ be the completion of $\dcal$ with respect to this norm. By Corollary 4.1.5 of~\cite{jn16}, $\sum_{N\ge0}c_N(T_1,\dotsc,T_{n-1})X^N$ can be computed by the action of $U_p$ on the space $\dcal^{1/p}(G,U_0(p))$. 

By Section 3.3 of~\cite{jn16}, $\dcal^r$ has a potential orthonormal basis given by the elements $e_{r,\eta}:=T_a^{-n(r,T_a,\eta)}\bar{\nfr}^\eta$, where 
\[
n(r,T_a,\eta)=\left\lfloor\frac{|\eta|\log_pr}{\log_p|T_a|}\right\rfloor,
\] 
and correspondingly $\dcal^r(G,U_0(p))$ has a potential orthonormal basis given by the elements
\[
e_{r,\eta}^t:=(0,\dotsc,0,e_{r,\eta},0,\dotsc,0)\subset\bigoplus_{t=1}^h \dcal^r\cong\dcal^r(G,U_0(p))
\]
where the $e_{r,\eta}$ is in the $t$th position. By Lemma 6.2.1 of~\cite{jn16}, we have
\[
U_p(e_{r,\eta}^t)=\sum_{u,\mu}a_\mu^u e_{r,\mu}^u
\]
with
\[
|a_\mu^u|\le|T_a|^{n(r,T_a,\mu)-n(r^{1/p},T_a,\mu)}.
\]
We have $n(p^{-1},T_a,\mu)=|\mu|$ and $n(p^{-1/p},T_a,\mu)=\flr{|\mu|/p}$. So whenever $|\mu|=N$, every matrix entry of $U_p$ in the row $e_{r,\mu}^u$ has coefficient $a_\mu^u$ divisible by $|T_a|^{N-\flr{N/p}}$. There are  
\[
h\binom{N+n(n-1)/2-1}{n(n-1)/2-1}
\]
choices of $u$ and $\mu$ such that $|\mu|=N$, and hence that many rows which we can guarantee are divisible by $T_a^{N-\flr{N/p}}$ (not counting rows which we can guarantee are divisible by higher powers of $T_a$). We conclude that $\np\left(\sum_{N\ge0}c_N(T_1,\dotsc,T_{n-1})X^N\right)$ passes above the point
\[
\left(h\sum_{N=0}^M\binom{N+n(n-1)/2-1}{n(n-1)/2-1},h\sum_{N=0}^M\binom{N+n(n-1)/2-1}{n(n-1)/2-1}(N-\flr{N/p})v(T_a)\right)
\]
for every integer $M\ge0$. Since the $x$-coordinate of the above expression is a polynomial in $M$ of degree $n(n-1)/2$ and the $y$-coordinate is $v(T_a)$ times a polynomial in $M$ of degree $n(n-1)/2+1$, the claim follows.
\end{proof}

\subsection{Systems of eigenvalues associated to classical points}
\label{myubdproof}

A ``refined principal series'' is a principal series representation $\pi$ of $GL_n(\rats_p)$ together with an ordered sequence of characters $(\psi_1,\dotsc,\psi_n):(\rats_p^\times)^n\to\cplx^\times$ such that $\pi\cong\pi(\psi_1,\dotsc,\psi_n)$. So there are $n!$ possible refinements of each $\pi$. The language comes from Galois representation theory. 
From our setup so far, it is easy to see that an eigenform $f\in (U^{sm}(\chi)\ten S_t)(G,U_0(p))$ is naturally associated to a particular refined principal series: the principal series $\pi_{f,p}$, together with, if $f$ has $\lambda$-values $\lam_1,\dotsc,\lam_n$, the ordered sequence $(\psi_1,\dotsc,\psi_n):(\rats_p^\times)^n\to\cplx^\times$ such that $\pi\cong\pi(\psi_1,\dotsc,\psi_n)$ and $\lam_i=p^{(n-1)/2-i+1-m_n-m_{n-1}-\dotsb-m_{n-i+1}}\psi_i(p)$. 
Also note that this refined principal series depends only on the point $x$ on $\dsc$ that $f$ is associated to.

For a character $\chi:(\ints_p^\times)^n\to\cplx^\times$ or $\psi:(\rats_p^\times)^n\to\cplx^\times$, and for any $w\in S_n$, we write $\chi^w=(\chi_{w(1)},\dotsc,\chi_{w(n)})$, and $\psi^w$ similarly.

Now note that if $f_x\in (U^{sm}(\chi)\ten S_t)(G,U_0(p))$ is an eigenform associated (via Proposition~\ref{fs-appear}) to a point $x$ on $\dsc$ with associated refined principal series $(\pi(\psi),\psi^{\id})$, then the refined principal series $(\pi(\psi),\psi^w)$ is also associated to a point $x^w$ on $\dsc$ and a form $f_x^w\in (U^{sm}(\chi^w)\ten S_t)(G,U_0(p))$ (arising from the unique $(J,\chi^w)$-vector in $\pi(\psi)$). The forms $f_x^w$ are called companion forms of $f_x$. Having defined these companion forms, it is straightforward to show that the slopes appearing in $(U^{sm}(\chi)\ten S_t)(G,U_0(p))$ are not only finite but bounded above by a linear function of $t$, as follows.

\begin{prop}
\label{ubd}
If $f\in (U^{sm}(\chi)\ten S_t)(G,U_0(p))$ is a $U_p^a$-eigenform with eigenvalue $a_p^{\id}$, 
and each companion form $f^w$ has $U_p^a$-eigenvalue $a_p^w$, then we have 
\[
\sum_{w\in S_n}v(a_p^w)=l^a(t)
\] 
where $l^a(t)$ is a linear function of $t_1,\dotsc,t_n$.

In particular, let $l^{(n-1,n-2,\dotsc,0)}(t)=l(t)$. Suppose that $\cond(\chi_i\chi_j^{-1})=\max(\cond(\chi_i),\cond(\chi_j))$ for all $i\neq j$. Then for each $w$, the Newton polygon of 
\[
\sum_{N\ge0}c_N(T_1(t\chi^w),\dotsc,T_{n-1}(t\chi^w))X^N
\] 
contains $hp^{j(\chi)}d_t$ slopes of size at most $l(t)$, hence in particular passes below the point 
\[
\left(hp^{j(\chi)}d_t, hp^{j(\chi)}d_tl(t)\right).
\]
\end{prop}


\begin{proof}
Let $\pi_{f,p}=\pi(\psi_1,\dotsc,\psi_n)$. By Proposition~\ref{lamtopsi}, we have
\[
\prod_i\lam_i=p^{-(nm_n+(n-1)m_{n-1}+\dotsb+m_1)}\prod_i\psi_i(p).
\]
The $\lam$-values of $x^w$ are given by
\[
\lam_i^w=p^{(n-1)/2-i+1-m_n-m_{n-1}-\dotsb-m_{n-i+1}}\psi_{w(i)}(p)
\]
for each $w\in S_n$. Then for $a=(a_1,\dotsc,a_n)$, the $U_p^a$-eigenvalue associated to $x^w$ is 
\[
\prod_i(\lam_i^w)^{a_{n-i+1}}=\prod_ip^{a_{n-i+1}[(n-1)/2-i+1-m_n-m_{n-1}-\dotsb-m_{n-i+1}]}\psi_{w(i)}(p)^{a_{n-i+1}}
\]
so the product of the $U_p^a$-eigenvalues associated to all the $x^w$s is
\[
p^{(n-1)!\sum_ia_{n-i+1}[(n-1)/2-i+1-m_n-m_{n-1}-\dotsb-m_{n-i+1}]}\left(\prod_i\psi_i(p)\right)^{(n-1)!\sum_ia_i}
\]
\[
=p^{(n-1)!\sum_ia_{n-i+1}[(n-1)/2-i+1-m_n-m_{n-1}-\dotsb-m_{n-i+1}]}\left(p^{nm_n+(n-1)m_{n-1}+\dotsb+m_1}\prod_i\lam_i\right)^{(n-1)!\sum_ia_i}.
\]
But $\prod_i\lam_i$ is the eigenvalue associated to the operator $U_p^{(1,1,\dotsc,1)}$, which is just right translation by the central matrix $\diag(p,p,\dotsc,p)$, which preserves $f$, so $\prod_i\lam_i=1$. So the sum of the valuations of the $U_p^a$-eigenvalues associated to the companion points is 
\[
(n-1)!\sum_ia_{n-i+1}[(n-1)/2-i+1-m_n-m_{n-1}-\dotsb-m_{n-i+1}]
\]
\[
+(nm_n+(n-1)m_{n-1}+\dotsb+m_1)(n-1)!\sum_ia_i
\]
\[
=(n-1)!\left(\sum_ia_{n-i+1}((n-1)/2-i+1)-\sum_j m_j(a_1+\dotsb+a_j)+\sum_j jm_j\left(\sum_i a_i\right)\right).
\]
Defining $l^a(t)$ to be this last expression, we find that $\sum_{w\in S_n}v(a_p^w)=l^a(t)$ as desired.

The conclusion that each individual $v(a_p^w)$ is bounded above by $l^a(t)$ follows because all the $a_p^w$s are algebraic integers.
\end{proof}

Let $c_i=\cond(\chi_i)$, let $\chi_{(1)},\dotsc,\chi_{(n-1)}$ be the characters $\chi_1,\dotsc,\chi_{n-1}$ reordered so that $\cond(\chi_{(1)})\le\cond(\chi_{(2)})\le\dotsb\le\cond(\chi_{(n-1)})$, let $c_{(i)}=\cond(\chi_{(i)})$, and let $T_{(i)}=T(\chi_{(i)})$. To get from Proposition~\ref{ubd} to the statement of Theorem~\ref{mybounds}, we just need to check that for all $t$ and $\chi$ such that $m_i\ge \eps m_j$ for all $i\neq j$ and $\cond(\chi_i\chi_j^{-1})=\max(\cond(\chi_i),\cond(\chi_j))$ for all $i\neq j$, $\left(hp^{j(\chi)}d_t, hp^{j(\chi)}d_tl(t)\right)$ has the desired numerical qualities. First we check the size of $d_t$.

\begin{prop}
The dimension $d_t$ is a polynomial of total degree $\frac{n(n-1)}2$ in $m_1,\dotsc,m_{n-1}$.
\end{prop}

\begin{proof}
By Corollary 14.9 of~\cite{ms04}, $\ind_{B(\ints_p)}^{GL_n(\ints_p),alg.}t$ has a basis indexed by chains in the poset described in Section 14.2 of~\cite{ms04}. For a subset $\sg$ of $\{1,\dotsc,n\}$, let $f(\sg)=\sum_{k\notin\sg}(n+1-k)$. We claim that when you take one step down the poset, $f(\sg)$ goes down by $1$. This is because, if $\sg$ is one step below $\tau$, there are two possibilities. The first is that $|\tau|=|\sg|$ and there is some $i$ for which $\sg_i=\tau_i-1$ and $\sg_j=\tau_j$ for all $j$ with $j\neq i$; in this case the complements $\sg^c$ and $\tau^c$ are the same except for $\sg_i\in\tau^c$ and $\sg_i+1=\tau_i\in\sg^c$, which contribute $n-\sg_i$ and $n-\sg_i-1$ to the sums $f(\sg)$ and $f(\tau)$, so $f(\sg)=f(\tau)-1$. The second is that $|\sg|=|\tau|+1$ and $\sg$ contains $n$ and $\tau$ does not, so again $f(\sg)=f(\tau)-1$. 

So a maximal chain in this poset starts with $\{n\}$, which has $f$-value $2+\dotsb+n=\frac{n(n+1)}2-1$, and ends with $\{1,2,\dotsc,n-1\}$, which has $f$-value $1$; its length is therefore $\frac{n(n+1)}2-1$. A leading term of $d_{m_1,\dotsc,m_{n-1},0}$ comes from distributing $m_1,\dotsc,m_{n-1}$ among corresponding variables in a maximal chain. So it is a product $\prod\binom{m_i+c_i}{c_i}$ where the $c_i+1$s sum to $\frac{n(n+1)}2-1$; that is, the $c_i$s sum to $\frac{n(n+1)}2-1-(n-1)=\frac{n(n-1)}2$. 
\end{proof}

Since $m_i\ge\eps m_j$ for all $i\neq j$, we can find some $A_{\eps}$ such that $l(t)\le A_{\eps}d_t^{\frac2{n(n-1)}}$ for all such $m_1,\dotsc,m_{n-1}$. Also, by the formula stated in Example~\ref{localgtcoords}, we have
\[
v(T_{(i)})=v(T(\chi_{(i)}))=Ap^{-c_{(i)}}
\]
for a constant $A$ (depending on $p$). Thus we have
\[
p^{j(\chi)}=p^{c_{(1)}+2c_{(2)}+\dotsb+(n-1)c_{(n-1)}-\frac{n(n-1)}2}=A'v(T_{(1)})^{-1}v(T_{(2)})^{-2}\dotsb v(T_{(n-1)})^{-(n-1)}.
\]
So if we let $x=hp^{j(\chi)}d_t$ and $y=hp^{j(\chi)}d_tl(t)$, we have
\begin{align*}
y&=\left(hp^{j(\chi)}d_t\right)^{1+\frac2{n(n-1)}}\left(hp^{j(\chi)}d_t\right)^{-\frac2{n(n-1)}}l(t) \\
&=Ax^{1+\frac2{n(n-1)}}\left(p^{j(\chi)}\right)^{-\frac2{n(n-1)}}d_t^{-\frac2{n(n-1)}}l(t) \\
&\le AA_{\eps}x^{1+\frac2{n(n-1)}}\left(v(T_{(1)})^{-1}v(T_{(2)})^{-2}\dotsb v(T_{(n-1)})^{-(n-1)}\right)^{-\frac2{n(n-1)}} \\
&=A'\left(v(T_{(1)})^{\frac{2}{n(n-1)}}v(T_{(2)})^{\frac{2\cdot 2}{n(n-1)}}\dotsb v(T_{(n-1)})^{\frac{2\cdot(n-1)}{n(n-1)}}\right)x^{1+\frac2{n(n-1)}}
\end{align*}
where $A'$ depends only on $n$, $p$, $h$, and $\eps$, as desired. This proves Part~\ref{myubd} of Theorem~\ref{mybounds}. 


\subsection{Combining upper bound points}
\label{combine}

We show that Theorem~\ref{unified-ubd} is a natural consequence of Part~\ref{myubd} of Theorem~\ref{mybounds}. First we need the following lemma of Wan, which is stated in~\cite{wan98} with $\ints_p$-coefficients but works identically with $\osc_{\cplx_p}$-coefficients.

\begin{lem}[Wan 1998]
\label{np-coincide}
Let $Q_1(X),Q_2(X)$ be two elements in $\osc_{\cplx_p}\ps{X}$ with $Q_1(0)=Q_2(0)=1$. Let $N_i(x)$ be the function on $\real_{\ge0}$ whose graph is the Newton polygon of $Q_i(X)$. Assume that $\nu(x)$ is a strictly increasing continuous function on $\real_{\ge0}$ such that $\nu(0)\le0$, $N_i(x)\ge x\nu(x)$ for $1\le i\le 2$ and $x\ge1$, and $\lim_{x\to\infty}\nu(x)=\infty$. Assume further that the function $x\nu^{-1}(x)$ is increasing on $\real_{>0}$, where $\nu^{-1}(x)$ denotes the inverse function of $\nu(x)$ defined at least on $\real_{\ge0}$. For $x\ge0$, we define the integer-valued increasing function $m_\nu(x)=\flr{x\nu^{-1}(x)}$. If the congruence
\[
Q_1(X)\equiv Q_2(X)\pmod{p^{m_\nu(\al)+1}}
\]
holds for some $\al\ge0$, then the two Newton polygons $N_i(x)$ coincide for all the sides with slopes at most $\al$.
\end{lem}


\begin{proof}[Proof of Theorem~\ref{unified-ubd}]
By Corollary~\ref{ubd}, $\np(t\chi)$ passes below the point
\[
\left(hp^{j(\chi)}d_t, hp^{j(\chi)}d_tl(t)\right).
\]
Note that the slope of $\np(t\chi)$ at $x$-coordinate $hp^{j(\chi)}d_t$ is at most $l(t)$. We may apply Lemma~\ref{np-coincide} with $\nu(x)=A_1x^{\frac2{n(n-1)}}\min_iv(T(\chi_i))$, so that 
\[
m_\nu(x)\asymp \frac{x^{1+\frac{n(n-1)}2}}{\left(\min_iv(T(\chi_i))\right)^{\frac{n(n-1)}2}}.
\]
Let $t_i^{(1)}=t_i+(n-i)p^{m_\nu(l(t))+1}\phi(q)$. By Lemma~\ref{np-coincide} applied to $P(X,t\chi)$ and $P(X,t^{(1)}\chi)$, we find that $\np(t^{(1)}\chi)$ also passes below this point. (The factor of $\phi(q)$ is to keep $t_1^{(1)},\dotsc,t_{n-1}^{(1)}$ in the same equivalence class as $t_1,\dotsc,t_{n-1}\pmod{\phi(q)}$ so that they fall in the same weight polydisc; presumably it would also suffice to twist by an appropriate tame character instead.) However, by Corollary~\ref{ubd}, $\np(t^{(1)}\chi)$ also passes below 
\[
\left(hp^{j(\chi)}d_{t^{(1)}}, hp^{j(\chi)}d_{t^{(1)}}l(t^{(1)})\right).
\]
Repeating this, we find a sequence $t=t^{(0)},t^{(1)},t^{(2)},\dotsc$ of dominant algebraic weights such that $\np(t^{(k)}\chi)$ passes below
\[
\left(hp^{j(\chi)}d_{t^{(0)}}, hp^{j(\chi)}d_{t^{(0)}}l(t^{(0)})\right),\dotsc,\left(hp^{j(\chi)}d_{t^{(k)}}, hp^{j(\chi)}d_{t^{(k)}}l(t^{(k)})\right).
\]
Evidently the $t^{(k)}$ approach a limit $t^\infty$, and $\np(t^\infty\chi)$ passes below
\[
\left(hp^{j(\chi)}d_{t^{(k)}}, hp^{j(\chi)}d_{t^{(k)}}l(t^{(k)})\right)
\]
for all $k$. The result follows as in the end of Section~\ref{myubdproof}. (Note that since $m_i^{(k)}=m_i^{(k-1)}+p^{m_\nu(l(t^{(k-1)}))+1}\phi(q)$, if $m_i^{(k-1)}\ge\eps m_j^{(k-1)}$ for all $i\neq j$, the same is true for the $m_i^{(k)}$.)
\end{proof}


\section{Geometry of the eigenvariety over the boundary of weight space}
\label{geometry}

Fix an index $a$, and let $\wsc_{<\nu}$ be the subset of characters $w$ such that $v(T_a(w))<\nu$ and $v(T_a(w))<\nu v(T_j(w))$ for all $j\neq a$ (so in particular $v(T_a)=\min_i v(T_i)$). let $\zsc_{<\nu}$ be the preimage of $\wsc_{<\nu}$ in the eigencurve $\zsc$. For any real number $\al$, let $X(<\al)$ be the subset of $\zsc$ of points $x$ for which $v(a_p(x))<\al v(T_a(w(x)))$, and define $X(=\al)$, $X(>\al)$ similarly.

As in the previous section, fix a polydisc in $\wsc$. For $T=(T_1,\dotsc,T_{n-1})$ in the polydisc and $m=(m_1,\dotsc,m_{n-1})\in\ints_{\ge0}^{n-1}$, write $T^m=T_1^{m_1}\dotsb T_{n-1}^{m_{n-1}}$ for short. Let 
\[
\det(1-XU_p)=\sum_{N\ge0}c_N(T)X^N,
\] 
where 
\[
c_N(T)=\sum_{m=(m_1,\dotsc,m_{n-1})\in\ints_{\ge0}^{n-1}}b_{N,m}T^m\in\ints_p\ps{T_1,\dotsc,T_{n-1}}.
\] 
Let $y=\np(T)(x)$ be the Newton polygon of $\sum_{N\ge0}c_N(T)X^N$.

For the following theorem, the only input we need is a lower bound for $y=\np(T)(x)$ of the form $y=v(T_a)f(x)$ where $f(x)$ is a convex function, which we have (with $f(x)=A_1x^{1+\frac2{n(n-1)}}$) from Part~\ref{mylbd} of Theorem~\ref{mybounds}.

\begin{thm}
\label{disconnect}
For every $\al\in\real_{\ge0}$, there is some valuation $\nu(\al)>0$ such that $X(=\al)_{<\nu(\al)}$ is disconnected from its complement in $\zsc_{<\nu(\al)}$.
\end{thm}


\begin{proof}
Let $d(\al,T)$ be the number of slopes in $y=\np(T)(x)$ of value strictly less than $\al v(T_a)$ (that is to say, the dimension of $\ssc_T(G,U_0(p))^{<\al v(T_a)}$). Assume $v(T_a)<1$. 

We claim that the point $(d(\al,T),\np(T)(d(\al,T)))$ lies inside the region bounded by the line $y=\al v(T_a)x$ and the function $y=v(T_a)f(x)$. It lies below $y=\al v(T_a)x$ because all slopes of $NP(T)$ up to $d(\al,T)$ are less than $\al v(T_a)$. It lies above $y=v(T_a)f(x)$ because this is a lower bound for $y=\np(T)(x)$.

This region lies inside the box whose lower left corner is $(0,0)$ and whose upper right corner is $(d(\al),\al d(\al)v(T_a))$, where $d(\al)$ is the nonzero solution to $\al x=f(x)$.

We have $(d(\al,T),\np(T)(d(\al,T)))=(j,v(c_j(T)))$ for some $j$. This is a vertex of $y=\np(T)(x)$. The vertex immediately preceding it is of the form $(i,v(c_i(T)))$ for some $i$. The slope between the two is
\[
\frac{v(c_j(T))-v(c_i(T))}{j-i}.
\]
This is the largest slope of $y=\np(T)(x)$ less than $\al v(T_a)$. We have $1\le j-i\le d(\al)$.

But $c_j(T)=\sum_{m\ge0}b_{j,m}T^m$ is a sum of terms $b_{j,m}T^m$ where $v(b_{j,m})$ is an integer and $v(T^m)=m_1v(T_1)+\dotsb+m_{n-1}v(T_{n-1})$. Thus $v(c_j(T))=\mu_j+\lam_j^1v(T_1)+\dotsb+\lam_j^{n-1}v(T_{n-1})$ where $\mu_j,\lam_j^k$ are integers in the range $[0,\al d(\al)]$ (since $v(c_j(T))\le \al d(\al)v(T_a)$). Similarly $v(c_i(T))=\mu_i+\lam_i^1v(T_1)+\dotsb+\lam_i^{n-1}v(T_{n-1})$ where $\mu_i,\lam_i^k\in[0,\al d(\al)]$ as well.

Assume that $v(T_a)<\frac1{\al d(\al)}$, so that $\al d(\al) v(T_a)<1$, and furthermore that $v(T_a)<\frac1{\al d(\al)}v(T_j)$ for all $j\neq a$. Then in order to have $v(c_i(T)),v(c_j(T))\le \al d(\al) v(T_a)$, we must have $\mu_i=\mu_j=0$ and $\lam^k=0$ for all $k\neq i$.

So the largest slope of $y=\np(T)(x)$ less than $\al v(T_a)$ is of the form $\frac{\lam_j-\lam_i}{j-i}v(T_a)$, where $\lam_j-\lam_i\in[0,\al d(\al)]$ and $j-i\in[1,d(\al)]$. This is a finite, discrete set of points. So the ratio of the largest slope of $y=\np(T)(x)$ less than $\al v(T_a)$ to $v(T_a)$ is bounded away from $\al$ independently of $T_a$. 

Setting $\nu(\al)<\frac1{\al d(\al)}$, we conclude that $X(<\al)_{<\nu(\al)}$ is disconnected from its complement in $\zsc_{<\nu(\al)}$.

This argument goes through exactly the same way if $X(<\al)$ is replaced by $X(\le\al)$: either the smallest slope greater than $\al$ is at least $\al+1$, or, if not, the next endpoint is again trapped in a box whose area is at most linear in $v(T)$, and the same argument applies. So we can choose $\nu(\al)$ such that $X(=\al)_{<\nu(\al)}$ is disconnected from its complement in $\zsc_{<\nu(\al)}$.
\end{proof}

As Liu-Wan-Xiao do in Theorem 3.19 of~\cite{lwx17}, we can also use Part~\ref{mylbd} of Theorem~\ref{mybounds} to give a simple proof of the fact that the ordinary part of $\zsc$ is finite and flat over $\wsc$ and disconnected from its complement.

\begin{thm}
$X(=0)$ is finite and flat over $\wsc$ and is a union of connected components of $\zsc$.
\end{thm}

\begin{proof}
The proof of Theorem 3.19 of~\cite{lwx17} goes through almost word-for-word. By Part~\ref{mylbd} of Theorem~\ref{mybounds}, there is some maximal $N$ such that $c_N(T_1,\dotsc,T_{n-1})$ is a unit in $\ints_p\ps{T_1,\dotsc,T_{n-1}}$, or equivalently, the constant term of $c_N(T_1,\dotsc,T_{n-1})$ is a unit in $\ints_p$. Then for each $(T_1,\dotsc,T_{n-1})$, the Newton polygon of $\sum_{n=0}^\infty c_N(T_1,\dotsc,T_{n-1})X^N$ starts with $N$ segments of slope $0$ followed by a segment of slope at least $\max(1,B\min_j v(T_j))$ for some constant $B$. Since $\max(1,B\min_j v(T_j))$ is uniformly bounded away from $0$ over any affinoid subdomain, $X(=0)$ is disconnected from its complement, and it is finite and flat of degree $N$. 
\end{proof}








\addcontentsline{toc}{section}{References}

\bibliography{refs}
\bibliographystyle{plain}

\end{document}